\documentclass[12pt,a4paper]{article}
\usepackage{color}
\usepackage[T1]{fontenc}
\usepackage[latin1]{inputenc}
\usepackage{amsmath}
\usepackage{amsthm}
\usepackage{amssymb}
\usepackage{enumitem}
\usepackage{stix}

\usepackage{url}
\usepackage{listings}

\numberwithin{equation}{section}

\usepackage{color}

\newcommand{\suppress}[1]{}
\newcommand{\R}{\mbox{$\mathbb{R}$}}

\newcommand{\N}{\mbox{$\mathbb{N}$}}

\newcommand{\jh}{\widehat{j}}
\newcommand{\iy}{\infty}

\newcommand{\undeux}{[n_1,\, n_2]}
\newcommand{\al}{\alpha}
\newcommand{\ga}{\gamma}
\newcommand{\ze}{\zeta}
\newcommand{\la}{\lambda}
\newcommand{\de}{\delta}
\newcommand{\si}{\sigma}
\newcommand{\om}{\omega}
\newcommand{\Om}{\Omega}
\newcommand{\eps}{\epsilon}
\newcommand{\vep}{\varepsilon}
\newcommand{\Nro}{N_{\rho}}
\newcommand{\Nrop}{N_{\rho}^+}
\newcommand{\ds}{\displaystyle}

\newcommand{\co}{\mathcal{O}}
\newcommand{\cp}{\mathcal{P}}
\newcommand{\ce}{\mathcal{E}}

\DeclareMathOperator{\limm}{li\text{$^{-1}$}}

\DeclareMathOperator{\li}{li}

\newcommand\Llr{\Longleftrightarrow}
\newcommand\lr{\longrightarrow}
\newcommand{\dsm}[1]{\mbox{$\displaystyle #1 $}}
\newcommand{\upon}[2]{\genfrac{}{}{0pt}{}{#1}{#2}}
\newcommand{\qtx}[1]{\quad\text{#1}\quad}
\newcommand{\Pplus}{\mathrm{P}^{+}}
\newcommand{\bigo}[1]{O\!\left(#1\right)}
\newcommand{\sk}{\sigma_{k}}
\newcommand{\bfni}[1]{\medskip\noindent{\bf#1}}

\newtheorem{theorem}{Theorem}[section]
\newtheorem{definition}[theorem]{Definition}
\newtheorem{prop}[theorem]{Proposition}
\newtheorem{coro}[theorem]{Corollary}
\newtheorem {lem}[theorem]{Lemma}

\newtheorem{rem}[theorem]{Remark}

\definecolor{deepblue}{rgb}{0,0,0.5}

\DeclareFixedFont{\ttm}{T1}{txtt}{m}{n}{9} 
\newcommand\pythonstyle{\lstset{
language=Python,
basicstyle=\ttm,
otherkeywords={self},   
keywordstyle=\ttm\color{deepblue},
emph={MyClass,__init__},          
emphstyle=\ttm\color{deepred},    
stringstyle=\color{deepgreen},
frame=tb,                         
showstringspaces=false            %
}}

\lstnewenvironment{python}[1][]
{
\pythonstyle
\lstset{#1}
}
{}

\newcommand\pythoninline[1]{{\pythonstyle\lstinline!#1!}}

\begin{document}
\title{The Landau function and the Riemann hypothesis}
\author{Marc Del\'eglise, Jean-Louis Nicolas\footnote{Research partially 
supported by CNRS, Institut Camille Jordan, UMR 5208.}}
\maketitle

\def\abstractname{Abstract}
\begin{abstract}
The Landau function $g(n)$ is the maximal order of an element  of 
the symmetric group ${\mathfrak S}_n$; it 
is also the largest product of powers of primes whose sum is $\le n$.
The main result of  this article is that the property 
``\,For all $n\ge 1$,  $\log g(n) < \sqrt{\limm (n)}\;$''  
(where $\limm$ denotes the inverse function of the logarithmic
integral)  is equivalent to the Riemann hypothesis.
\end{abstract}



\section{Introduction}
Let $n$ be a positive integer. In \cite{Lan}, Landau 
 introduced the function $g(n)$ as the maximal order of
an element in the symmetric group ${\mathfrak S}_n$; he showed that
\begin{equation}\label{g}
g(n) = \max_{\ell(M) \le n} M
\end{equation}
where $\ell$ is the additive function such that $\ell(p^{\al}) = p^{\al}$ 
for $p$ prime and  $\al \ge 1$. In other words, if the standard
factorization of $M$ 
is $M = q_1^{\al_1} q_2^{\al_2} \cdots q_j^{\al_j}$
we have $\ell(M) = q_1^{\al_1} + q_2^{\al_2} + \cdots + q_j^{\al_j}$
and $\ell(1) = 0$. He also proved that
\[
  \log g(n) \sim \sqrt{n\log n}, \quad n\to\iy.
\]
A function close to the Landau function is the function $h(n)$ defined 
for $n \ge 2$ as the greatest product of a family
of primes $q_1 < q_2 < \cdots < q_j$ the sum of which does not exceed $n$.
If $\mu$  denotes the M\"{o}bius function, $h(n)$
can also be defined by
\begin{equation}\label{h}
h(n) = \max_{\upon{\ell(M) \le n}{ \mu(M) \ne 0}} M.
\end{equation}
The above equality implies $h(1)=1$. Note that
\begin{equation}\label{ellh}
\ell(g(n)) \le n \qtx{and} \ell(h(n)) \le n.
\end{equation}
From \eqref{h} and \eqref{g}, it follows that
\begin{equation}\label{h<g}
 h(n) \le g(n), \quad (n \ge 1).
\end{equation}

Sequences $(g(n))_{n\ge 1}$ and $(h(n))_{n \ge 1}$  are sequences
A000793 and A159685  in the OEIS ({\it On-line 
Encyclopedia of Integer Sequences}). One can find results about $g(n)$
in \cite{MNRAA,MNRMC,DN,DNZ,MPE}, see also \cite{Mil} 
and \cite[\S 10.10]{Brou}.
In the introductions of \cite{DN,DNZ}, other references are given.
The three papers \cite{DNh1,DNh2,DNh3} are devoted to $h(n)$.
A fast algorithm to compute $g(n)$ (resp. $h(n)$) is described in \cite{DNZ}
(resp. \cite[\S 8]{DNh1}).
In \cite[(4.13)]{DNh2}, it is shown that
\begin{equation}
  \label{h<g568}
\log h(n) \le \log g(n) \le \log h(n)+5.68\;(n\log n)^{1/4},\quad
n\ge 1.
\end{equation}
Let $\li$ denote the logarithmic integral and $\limm$ its inverse
function (cf. below \S \ref{parlogint}). In \cite[Theorem 1
(i)]{MNRAA}, it is proved that
\begin{equation}
 \label{g=li}
\log g(n) = \sqrt{\limm (n)} +\co\left(\sqrt n \exp(-a\sqrt{\log n})\right)
\end{equation}
holds for some positive $a$. The asymptotic
expansion of $\log g(n)$ does coincide with the one of
$\sqrt{\limm(n)}$ (cf. \cite[Corollaire, p. 225]{MNRAA}) and also,
from \eqref{h<g568}, with the one of $\log h(n)$\,:
\begin{multline}
  \label{asygn}
\left. 
  \begin{array}{r}
\log h(n)\\
\log g(n)\\
\sqrt{\limm(n)}\\
  \end{array}
\right\}=\\
\sqrt{n\log n}\left(1+\frac{\log \log n -1}{2 \log n}
-\frac{(\log\log n)^2-6\log\log n+9+o(1)}{8\log^2 n}\right)
\end{multline}

In \cite[Th\'eor\`eme 1 (iv)]{MNRAA}, it is proved that under the Riemann hypothesis the
inequality  
\begin{equation}
  \label{g<li}
\log g(n) < \sqrt{\limm (n)}
\end{equation}
holds for $n$ large enough. In August
2009, the second author received an e-mail of Richard Brent asking
whether it was possible to replace ``$n$ large enough'' by ``$n\ge
n_0$''
with a precise value of $n_0$. The aim of this paper is to
anwer this question positively. 
For $n \ge 2$, let us introduce the sequences
\begin{align}
  \label{an}
\log g(n) = \sqrt{\li^{-1}(n)} - a_n (n \log n)^{1/4}
 \qtx{i.e.}
  a_n=\frac{\sqrt{\limm (n)}-\log g(n)}{(n\log n)^{1/4}},
\\
\label{bn}
\log h(n) = \sqrt{\li^{-1}(n)} - b_n (n \log n)^{1/4}
 \qtx{i.e.}
  b_n=\frac{\sqrt{\limm (n)}-\log h(n)}{(n\log n)^{1/4}},
\end{align}
and the constant
\begin{equation}
  \label{c}
  c=\sum_\rho \frac{1}{|\rho(\rho+1)|}= 0.046\,117\,644\,421\,509\ldots
\end{equation}
where $\rho$ runs over the non trivial zeros of  the Riemann $\ze$
function. The computation of the above 
numerical value is explained in \cite[Section 2.4.2]{DNh3}.
We prove

\begin{theorem}\label{thmgnHR}
Under the Riemann hypothesis, 
\begin{enumerate}[label=(\roman*)]
  \item
    $\ds \log g(n) <\sqrt{\limm (n)}$ for $n\ge 1$,
\item
  $\ds a_n \ge \frac{2-\sqrt 2}{3}-c-\frac{0.43\;\log \log n}{\log  n}
  > 0$
 \;\; for $n\ge 2$,
\item
  $\ds a_n \le \frac{2-\sqrt 2}{3}+c+\frac{1.02\;\log \log n}{\log n}$
  \;\; for $n\ge 19425$,
\item
  $\ds 0.11104 < a_n \le a_{2}=0.9102\ldots$ \quad \quad for $n\ge 2$,
\item
  $\ds 0.149\ldots = \frac{2-\sqrt 2}{3}-c \le \liminf a_n \le 
\limsup a_n \le \frac{2-\sqrt 2}{3}+c=0.241\ldots$ 
\item
  When $n\to \iy$,
  \begin{multline*}
 \ds \Big(\frac{2-\sqrt 2}{3}-c\Big)\Big(1+
 \frac{\log \log n+\co(1)}{4\log n}\Big) \le a_n \\
 \le \Big(\frac{2-\sqrt 2}{3}+c\Big)\Big(1+\frac{\log \log
  n+\co(1)}{4\log n}\Big).
\end{multline*}
\end{enumerate}
\end{theorem}

\begin{rem}
It does not seem easy to calculate $\inf_{n\ge 2} a_n$, and to decide
whether it is a minimum or not.  
\end{rem}

\begin{coro}\label{coroghnHR}
Each of the six points of Theorem \ref{thmgnHR} is equivalent to the
Riemann hypothesis.
\end{coro}

\begin{proof}
If the Riemann hypothesis fails, 
it is proved in \cite[Theorem 1 (ii)]{MNRAA} that there
exists $b > 1/4$ such that
\begin{equation}
  \label{gnxi}
  \log g(n) =\sqrt{\limm (n)} +\Om_{\pm}((n\log n)^b)
\end{equation}
which contradicts (i), (ii), $\ldots$, (vi) of Theorem \ref{thmgnHR}.
\end{proof}

In the paper \cite{DNh3}, the following theorem is proved:

\begin{theorem}\label{thmhnHR}
Under the Riemann hypothesis, 
\begin{enumerate}[label=(\roman*)]
 \item
$\ds \log h(n) <\sqrt{\limm (n)}$ for $n\ge 1$,
\item
$\ds b_{17}=0.49795\ldots \le b_n \le b_{1137}=1.04414\ldots$ for $n\ge
2$,
\item
$\ds b_n \ge \frac 23-c-\frac{0.23\;\log \log n}{\log n}$ for $n\ge 18$,
\item
$\ds b_n \le \frac 23+c+\frac{0.77\;\log \log n}{\log n}$ for
$n\ge 4\ 422\ 212\ 326$,
\item
$\ds 2/3-c =0.620\ldots\le \liminf b_n \le 
\limsup b_n \le 2/3+c=0.712\ldots$ 
\item
  and, when $n\to \iy$,
  \begin{multline*}
 \ds \left(\frac 23-c\right)\left(1+
   \frac{\log \log n+\co(1)}{4\log n}\right) \le b_n
 \\ \le \left(\frac 23+c\right)\left(1+\frac{\log \log
     n+\co(1)}{4\log  n}\right).
\end{multline*}
\end{enumerate}
\end{theorem}

The main tools in the proof of Theorem \ref{thmhnHR} in \cite{DNh3} are the explicit
formulas for $\sum_{p^m\le x} p$ and $\sum_{p^m\le x} \log p$. 

We deduce Theorem
\ref{thmgnHR} about $g(n)$ from Theorem \ref{thmhnHR} about $h(n)$ by
studying the difference $\log g(n) -\log h(n)$ in view of 
improving inequalities \eqref{h<g568}. More precisely,
we prove

\begin{theorem}\label{thmgsh}
Without any hypothesis, 
\begin{enumerate}[label=(\roman*)]
\item
  For $n$ tending to infinity,
  \begin{multline*}
    \log \frac{g(n)}{h(n)} =  \frac{\sqrt 2}{3}(n\log n)^{1/4}
    \Bigg(1+\frac{\log \log n-4\log 2-11/3}{4\log n}\\
  -\frac{\frac{3}{32} (\log\log n)^2
      -\left(\frac{3\log 2}{4}+\frac{15}{16}\right)\log\log n+
      \frac{(\log 2)^2}{2}+\frac{29\log 2}{12}+\frac{635}{288}}{(\log n)^2}\Bigg)\\
  + \co\left(\frac{(\log \log n)^3}{(\log n)^3}\right)\quad
\end{multline*}
\item
\begin{multline*}\log \frac{g(n)}{h(n)} \le \frac{\sqrt 2}{3}(n\log n)^{1/4}
\left(1+\frac{\log \log n + 2.43}{4 \log n}\right)\\
\textrm{ for } n\ge 3\ 997\ 022\ 083\ 663,
\end{multline*}
\item
$\ds \log \frac{g(n)}{h(n)} \ge \frac{\sqrt 2}{3}(n\log n)^{1/4}
\left(1+\frac{\log \log n-11.6}{4 \log n}\right)$ for $n\ge 4\,230$,
\item
For $n\ge 1$ we have $\dfrac{g(n)}{h(n)}\ge 1$  with equality for\\ 
$n=1,2,3,5,6,8,10,11,15,17,18,28,41,58,77$.
\item
  For $n \ge 1$, we have  $\log \dfrac{g(n)}{h(n)} \le 0.62066\dots
  (n\log n)^{1/4}$ with equality for $n=2243$.
\end{enumerate}
\end{theorem}

\begin{rem}
From the asymptotic expansion (i), it follows that, for $n$ very large,
the inequality $\log (g(n)/h(n)) > (\sqrt 2/3)(n\log n)^{1/4}$
holds. But finding the largest $n$ for which \\
$(\log (g(n)/h(n)))/(n \log n)^{1/4}$ does
not exceed  $\sqrt 2/3 = 0.471\ldots$ seems difficult.
\end{rem}

\begin{theorem}\label{thmhgMm}
For $n\geq 373\,623\,863$, 
\begin{equation}\label{minhn}
\sqrt{n\log n}\left(1+\frac{\log \log n -1}{2 \log n}-
\frac{(\log\log n)^2}{8\log^2 n}\right)\le\log h(n) \le\log g(n)
\end{equation}
and, for $n\geq 4
$,
\begin{equation}\label{maxgn}
\log h(n) \le \log g(n) \le\sqrt{n\log n}\left(1+\frac{\log \log n -1}{2 \log n}\right)
\end{equation}
\end{theorem}

The lower bound \eqref{minhn} of $\log h(n)$ improves on 
\cite[Theorem 4]{DNh2} where it was shown that, for $n\ge 77\, 615\,
268$, we have $\log h(n) \ge \sqrt{n\log n}\left(1+\frac{\log \log n -1.16}{2 \log n}\right)$.
Inequality \eqref{maxgn} improves on the result of 
\cite[Corollaire, p. 225]{MNRAA},  where $-1$ was replaced by $-0.975$. 
From the common asymptotic
expansion \eqref{asygn} of $\log h(n)$  and $\log g(n)$, one can see
that the constants $8$ in \eqref{minhn} and $-1$ in \eqref{maxgn}  
are  optimal.

\subsection{Notation} 

\begin{itemize}
 \item
$\ds \pi_r(x)=\sum_{p\le x} p^r$. For $r=0$,
$\ds \pi(x)=\pi_0(x)=\sum_{p\le x} 1$ is the prime counting function.
$\ds \pi_r^-(x)=\sum_{p< x} p^r$.
 \item
$\ds\theta(x)=\sum_{p\le x} \log p$ is the Chebichev function.
$\theta^-(x)=\sum_{p< x} \log p$.
  \item             
  $\cp=\{2,3,5,\ldots\}$ denotes the set of primes. 
$p_1=2,\  p_2=3,\ldots, p_j$ is the $j$-th prime. For $p\in \cp$ 
and $n\in \N$,  $v_p(n)$ denotes the largest exponent such that 
$p^{v_p(n)}$ divides $n$.
\item $\Pplus(n)$ denotes the largest prime factor of $n$.
\item             
$\ds\li(x)$  denotes the logarithmic integral of $x$ 
(cf. below Section \ref{parlogint}). The inverse function is denoted by 
$\limm$.
\item
If $\ds \lim_{n\to\iy} u_n=+\iy$, $v_n=\Om_{\pm}(u_n)$ is equivalent to
$\ds \limsup_{n\to\iy} v_n/u_n >0$ and $\ds \liminf_{n\to\iy} v_n/u_n <0$.
\item
We use the following constants: 
\begin{itemize}
\item
$x_1$ takes three values (cf. \eqref{dusart3}),
\item
$x_0=10^{10}+19$ is the smallest prime exceeding $10^{10}$,
$\log(x_0)=23.025\ 850\ldots$,
\item
$\nu_0=2\ 220\ 832\ 950\ 051\ 364\ 840=2.22\ldots 10^{18}$ is defined
below in \eqref{n0},
\item
$\log \nu_0 = 42.244 414\ldots,\log \log \nu_0=3.743 472\ldots$
\item
The numbers $(\la_j)_{j \ge 2}$ described in Lemma \ref{lemxik}
and  $(x_j^{(0)})_{2 \le j \le 29}$ defined in \eqref{xj0}.
\item
For convenience, we sometimes write $L$ for $\log n$, $\la$ for $\log
\log n$, $L_0$ for $\log \nu_0$ and $\la_0$ for $\log \log \nu_0$. 

\end{itemize}
\end{itemize}
\medskip
       
We often implicitly use the following result~: for $u$ and $v$
positive and $w$ real, the function
\begin{equation}\label{fab}
t\mapsto \frac{(\log t-w)^u}{t^v} \quad \text{is decreasing for}\quad t > \exp(w+u/v).
\end{equation}

Also, if $\eps$ and $\eps_0$ are  real numbers satisfying $0\le \eps
\le \eps_0 <1$\ we shall use the following upper bound
\begin{equation}
   \label{1s1meps}
\frac{1}{1-\eps}= 1+ \frac{\eps}{1-\eps}\le   1+ \frac{\eps}{1-\eps_0}.
\end{equation}

Let us write $\sigma_0 = 0$, $N_0 =
1$, and, for $j \ge 1$, 
\begin{equation}\label{Nj}
N_j = p_1 p_2 \cdots p_j
\qtx{ and }
\sigma_j = p_1 + p_2 + \cdots + p_j = \ell(N_j).
\end{equation}
For $n \ge 0$, let $k = k(n)$ denote the integer $k \ge 0$ such that 
\begin{equation}\label{kn}
\sigma_{k} = p_1 + p_2 + \cdots + p_{k} \le n < p_1 + p_2 + \cdots +
p_{k+1} = \sigma_{k+1}.
\end{equation}
In \cite[Proposition 3.1]{DNh1}, for $j\ge 1$, it is proved that
\begin{equation}
  \label{hsigk}
 h(\si_j)=N_j. 
\end{equation}
In the general case, one writes $n=\si_{k}+m$ with $0<m<p_{k+1}$ and,
from \cite[Section 8]{DNh1}, we have
\begin{equation}
  \label{hNG}
h(n)=N_k G(p_k,m)
\end{equation}
where $G(p_k,m)$ can be calculated by the algorithm described in \cite[Section 9]{DNZ}.

\subsection{Plan of the article}\label{parplan} 

\begin{itemize}
\item
In Section \ref{parUsRes}, we recall some effective bounds for 
the Chebichev function $\theta(x)$ and for $\pi_r(x)=\sum_{p\le x} p^r$.
We give also some properties of the logarithmic integral
$\li(x)$ and  its inverse $\limm$.

\item
Section \ref{parSupCh} is devoted to the definition and properties of
$\ell$-superchampion numbers. These numbers, defined on the model of
the {\it superior highly composite numbers} introduced by Ramanujan in
\cite{Ram}, are crucial for the study of the Landau function. They
allow the construction of an infinite number of integers $n$ for which
$g(n)$ is easy to calculate.
To reduce the running time of computation,
an argument of convexity is given in Section \ref{parcon} and used in Section
\ref{parproof174} and in Lemma \ref{lemgncon1939} in conjunction with the tools
presented in Section \ref{parComputationalPoints}.
\item In Section \ref{parComputationalPoints} we present some methods used
  to compute efficiently: how to quickly enumerate the
  superchampion numbers,
how  to find the largest integer $n$ in a finite interval
  $[a,b]$ which does not satisfy a boolean property $ok(n)$, by
  computing only a small number of values $ok(n)$.
\item
  In Section \ref{parProofthmhgMm}  we prove  Theorem \ref{thmhgMm}.

\item In Section \ref{parLargen}, in preparation to the proof of
  Theorem \ref{thmgsh}, we study the function $\log g(n)-\log h(n)$
  for which we give an effective estimate for $n \ge \nu_0$  (defined
  in \eqref{n0}) and also an  asymptotic estimate.

\item
In Section \ref{parproofthmgsh}, we prove Theorem \ref{thmgsh},
first for $n\ge \nu_0$, by using the results of
Section \ref{parLargen}, and further, for $n < \nu_0$, by explaining the required
computation.
\item
  In Section \ref{parproofthmgn}, we prove
 Theorem \ref{thmgnHR}. For $n\ge \nu_0$, it follows
from the reunion of the proofs of Theorem \ref{thmhnHR} (given in
\cite{DNh3}) and of  Theorem \ref{thmgsh}. For $n <  \nu_0$, some more
computation is needed.
\end{itemize}
\bigskip
All computer calculations have been implemented in Maple and C$^{++}$.
Maple programs are slow but can be executed by anyone disposing of
Maple.  C$^{++}$ programs are much faster. They use real double
precision, except for the demonstration of the Lemma
\ref{lemgncon1939} where we used the GNU-MPFR Library to compute with
real numbers with a mantissa of 80 bits.  The most expensive
computations are the proof of theorem 1.5.(ii) and the proof of
Lemma \ref{lemgncon1939} which took respectively
40 hours and  10 hours of CPU (with the C$^{++}$ programs).
The Maple programs can be loaded on \cite{web}.

\section{Useful results}\label{parUsRes}

\subsection{Effective estimates}\label{parEffEst}

In \cite{But}, B\"uthe has proved
\begin{equation}\label{thx<x}
\theta(x) =\sum_{p\le x} \log p < x\;\; \text{ for } x \le 10^{19}
\end{equation}
while Platt and Trudgian in \cite{Pla} have shown that 
\begin{equation}\label{thx<x1}
\theta(x) <(1+7.5\cdot 10^{-7})\, x \text{ for } x \ge 2
\end{equation}
so improving on results of Schoenfeld \cite{Sch76}. 
Without any hypothesis, we know that
\begin{equation}
  \label{dusart3}
 |\theta(x)-x] < \frac{\al\  x}{\log^3 x}  \text{ for } x \ge x_1=x_1(\al)
 \end{equation}
with
\[\al= \left\{
\begin{array}{llrl}
1 & \text{ and} \quad x_1=&89\,967\,803 & \text{ (cf. \cite [Theorem 4.2]{Dus3}) }\\
0.5 & \text{ and} \quad x_1=&767\,135\,587  &\text{ (cf. \cite [Theorem 4.2]{Dus3}) }\\
0.15 & \text{ and} \quad x_1=&19\,035\,709\,163 & \text{ (cf. \cite[Theorem 1.1]{axl2}) }.\\ 
\end{array} \right.
\]
\begin{lem}\label{lemtheta}
  Let us denote  $\theta^-(x)=\sum_{p<x} \log p$. Then
  \begin{align}
   \label{eq7461}
   \theta(x) \ge  \theta^-(x)\  &\ge  x-0.0746\; \frac{x}{\log x}  & (x > 48757 )\\
   \label{eq79}
                    \theta(x) \     &\le
                     x\Big(1+\frac{0.000079}{\log x}\Big)  &  (x >1)   \\
  \label{pi126}
                    \pi(x)\  &<      1.26\frac{x}{\log x}  & (x > 1).
  \end{align}
\end{lem}

\begin{proof}
  \begin{itemize}
  \item
From \cite[Corollary 2*, p. 359]{Sch76}, for $x \ge 70877$, we have 
$\theta(x) >F(x)$ with $F(t)=t( 1-1/(15 \log t))$. As $F(t)$ is
increasing for $t >0$,  this implies that if
$x> 70877$ then $\theta^-(x) \ge F(x)$ holds. Indeed, if $x$ is not prime,
we have $\theta^-(x)=\theta(x)$ while if $x>70877$ is prime then
$x-1>70877$ holds and we have
$\theta^-(x)=\theta(y)> F(y)$ for $y$ satisfying $x-1 < y < x$.
When $y$ tends to $x$, we get $\theta^-(x)\ge F(x)$ and, as $1/15 < 0.0746$
holds, this proves \eqref{eq7461} for $x > 70877$. 
Now, let us assume
that $48757 < x \le 70877$ holds. For all primes $p$ satisfying
$48757 \le p < 70877$ and $p^+$ the prime following $p$ we consider
the function $f(t)=t(1-0.0746 /\log t)$ for $t\in (p,p^+]$. 
As $f$ is increasing, the maximum of $f$ is
$f(p^+)$ and $\theta^-(t)$ is constant and equal to $\theta(p)$. So, to
complete the proof of  \eqref{eq7461}, we check
that $\theta(p) \ge  f(p^+)$ holds for all these $p$'s.

\item
 \eqref{eq79} follows from
\eqref{thx<x} 
for $x \le 10^{19}$, while, for $x > 10^{19}$, from \eqref{dusart3}, we have
\begin{multline*}
\theta(x) \le x\left(1+\frac{0.15}{\log^3 x}\right) \le
x\left(1+\frac{0.15}{(\log^2 10^{19})(\log x)}\right)
\\ = x\left(1+\frac{0.0000783\ldots}{\log x}\right).
\end{multline*}
\item
  \eqref{pi126} is stated in \cite[(3.6)]{RS62}.
  \end{itemize}
\end{proof}

\begin{lem}\label{lemW}
Let us set
\begin{equation}\label{Wx}
W(x)=\sum_{p\le x}\frac{\log p}{1-1/p}.
\end{equation} 
Then, for $x>0$,
\begin{equation}\label{Wx7}
\frac{W(x)}{x}\le  \omega=
\begin{cases}
W(7)/7=1.045\,176\ldots& \text{ if } x\le 7.32\\
 1.000014& \text{ if } x > 7.32. 
\end{cases}
\end{equation} 
\end{lem}

\begin{proof}
First, we calculate $W(p)$ for all primes $p< 10^6$. 
For $11\le p<10^6$, $W(p) < p$ holds while, for $p\in \{2,3,5,7\}$,
the maximum of $W(p)/p$ is
attained for $p=7$. 
If $p$ and $p^+$ are consecutive primes, $W(x)$ is constant and $W(x)/x$ is
decreasing on $[p,p^+)$. As $W(7)=7.316\ldots$
this proves \eqref{Wx7} for $x\le 7.32$ and $W(x)< x$ for $7.32<x\le 10^6$.

Let us assume now that $x >y=10^6$ holds. We have
\begin{multline*}
  W(x)=W(y)+\sum_{y < p\le x} \frac{\log p}{1-1/p} \le
  W(y)+\frac{y}{y-1} \sum_{y<p\le x} \log p\\
=W(y) -\frac{y}{y-1}\theta(y)+\frac{y}{y-1}\theta(x)=12.240\, 465\ldots+\frac{10^6}{10^6-1}\theta(x)
\end{multline*}
and, from \eqref{thx<x1},
\[W(x)\le 12.241\frac{x}{10^6}+\frac{10^6}{10^6-1}(1+7.5\times 10^{-7})x <
  1.00001399\ldots\,x,\]
which completes the proof of Lemma \ref{lemW}.
\end{proof}

\begin{lem}\label{lempixy}
Let $K \ge 0$ and $\al >0$ be two real numbers. Let us assume that 
there exists $X_0 > 1$ such that, for $x \ge X_0$,
\begin{equation}
  \label{thK}
x-\frac{\al x}{\log^{K+1} x} \;\le \theta(x) \;  
\le \; x+\frac{\al x}{\log^{K+1} x}.   
\end{equation}
If $a$ is a positive real number satisfying $a < \log^{K+1} X_0$, 
for $x\ge X_0$, we have
\begin{equation}
  \label{pixyK}
  \pi\left(x+\frac{a\, x}{\log^K x}\right)-\pi(x) \ge b\;\frac{x}{\log^{K+1} x}
  \end{equation}
with
\[b=\left(1-\frac{a}{\log^{K+1} X_0} \right)
\left(a-\frac{2\al}{\log X_0} -\frac{\al\,a}{\log^{K+1} X_0} \right).\]
\end{lem}

\begin{proof}
Let us set $y=x(1+a/\log^K x)$. For $x\ge X_0$, we have
\begin{equation*}
1<X_0 \le x <  y=x\left(1+\frac{a}{\log^K x}\right)\le x\left(1+\frac{a}{\log^K X_0}\right)
\end{equation*}
and
\begin{multline}
  \label{logx<logy}
0 < \log x < \log y < \log x + \frac{a}{\log^K x}= 
(\log x)\left(1+\frac{a}{\log^{K+1} x}\right)\\
\le (\log x)\left(1+\frac{a}{\log^{K+1} X_0}\right)<  
\frac{\log x}{1-a/\log^{K+1} X_0}.
\end{multline}
Further,
\[\pi(y)-\pi(x) =\sum_{x<p\le y} 1\ge \sum_{x<p\le y}\frac{\log p}{\log y}=
\frac{1}{\log y}(\theta(y)-\theta(x))\]
and, from \eqref{thK}, 
\begin{multline*}
\pi(y)-\pi(x) \ge \frac{1}{\log y}
\left(y-x-\frac{\al \,y}{\log^{K+1} y}-\frac{\al\, x}{\log^{K+1} x}\right)
\\= \frac{1}{\log y}\left(\frac{a\,x}{\log^K x}-
\frac{\al\,y}{\log^{K+1} x}-\frac{\al\,x}{\log^{K+1} x}\right)\\
 =\frac{x}{(\log y)\log^K x}
 \left(a-\frac{2\al}{\log x}-\frac{\al\, a}{\log^{K+1} x}  \right)
\\  \ge \frac{x}{(\log y)\log^K x}
\left[a-\frac{2\al}{\log X_0}-\frac{\al\, a}{\log^{K+1} X_0}  \right].
\end{multline*}
If the above bracket is $\le 0$ then $b$ is also $\le 0$ and
\eqref{pixyK} trivially holds. If the bracket is positive then
\eqref{pixyK} follows from \eqref{logx<logy},
which ends the proof of Lemma \ref{lempixy}.
\end{proof}

\begin{coro}\label{coropixy}
For $x\ge x_0=10^{10}+19$, 
\begin{equation}
  \label{pixy}
  \pi(x(1+0.045/\log^2 x))-\pi(x) \ge 0.012 \sqrt x.
\end{equation}
\end{coro}

\begin{proof}
Since, for $x\ge x_0$, \eqref{dusart3} implies \eqref{thK}
with $\al=1/2$, $K=2$ and $X_0=x_0$, we may apply Lemma \ref{lempixy} that yields
$\pi(x(1+0.045/\log^2 x))-\pi(x) \ge b\,x/\log^3 x$ with
$b=0.001568\ldots$

From \eqref{fab}, for
$x\ge x_0$, $\sqrt x/\log^3 x$ is increasing and 
$\sqrt{x_0}/\log^3 x_0=8.19\ldots$, so that
$\pi(x(1+0.045/\log^2 x))-\pi(x) \ge b\,x/\log^3 x\ge 8.19\,b\, \sqrt
  x\ge 0.012 \,\sqrt x$.
\end{proof}

\subsection{The logarithmic integral}\label{parlogint}

For $x$ real $> 1$, we define $\li(x)$ as (cf. \cite[p. 228]{Abr})
\[\li (x) =\intbar_0^{x}  \frac{dt}{\log t}=\lim_{\vep\to 0^+}
\left(\int_0^{1-\vep}+\int_{1+\vep}^x \frac{dt}{\log t}\right)=\int_2^x \frac{dt}{\log t} +\li(2).\]
From the definition of $\li(x)$, it follows that
\begin{equation*}
\frac{d}{dx} \li(x)=\frac{1}{\log x}\quad \text{ and } \quad
\frac{d^2}{dx^2} \li(x)=-\frac{1}{x \log^2 x}.
\end{equation*}
For $x\to \iy$, the logarithmic integral has the asymptotic expansion
\begin{equation}
  \label{lixinfini}
\li(x)=\sum_{k=1}^N \frac{(k-1)!x}{(\log x)^k}+\co\left(\frac{x}{(\log x)^{N+1}}\right).
  \end{equation}
The function $t\mapsto \li(t)$ is an increasing bijection from
$(1,+\iy)$ onto $(-\iy,+\iy)$. 
We denote by $\limm(y)$ its 
inverse function that is defined for all $y\in \R$. Note that $\limm(y) > 1$
 always holds.

To compute numerical values of $\li(x)$, we used the 
formula, due to Ramanujan (cf. \cite[p. 126-131]{RamLi}),

\[
\li(x)=\ga_0+\log \log x  + \sqrt x \sum_{n=1}^\infty a_n (\log x)^n
\qtx{ with }
a_n = \frac{(-1)^{n-1}}{n!\, 2^{n-1}}
\sum_{m=0}^{\lfloor\frac{n-1}{2}\rfloor} \frac{1}{2m+1}.
\]
The computation of $\limm{y}$ is carried out by solving the equation
$\li(x)=y$ by the Newton method.

\subsection{Study of $\pi_r(x)=\sum_{p\le x} p^r$}\label{parpir}

In the article \cite{DNh3}, we have deduced from \eqref{dusart3} the
following proposition:

\begin{prop}\label{proppir}
Let $\al$, $x_1=x_1(\al)$  be two real numbers such that 
$0 < \al \le 1$,   $x_1 \ge 89\,967\,803$ and
$|\theta(x)-x| < \al\,x/\log^3 x$  for  $x\ge x_1$.
Then, for $r\ge 0.6$ and $x\ge x_1$,  
\begin{multline}\label{Majpir}
\pi_r(x) \le C_0+
\frac{x^{r+1}}{(r+1)\log x}+\frac{x^{r+1}}{(r+1)^2\log^2\! x}+
\frac{2x^{r+1}}{(r+1)^3\log^3\! x}\\
  +\frac{(51\al r^4+176 \al  r^3+222 \al  r^2+120
    \al  r +23 \al  +168)x^{r+1}}{24(r+1)^4 \log^4 x}
\end{multline}
with
\begin{multline}
\label{C0maj}
C_0 = \pi_r(x_1)-\frac{x_1^r\theta(x_1)}{\log x_1}
- \frac{3\al\,r^4+8\al\,r^3+6\al\,r^2+24-\al\,}{24} \li(x_1^{r+1})\\
   +\frac{(3 \al\, r^3+5 \al\, r^2+\al\, r+24-\al\,)x_1^{r+1}}{24\log x_1}
 \\  +\frac{\al\, (3r^2+2r-1)x_1^{r+1}}{24\log^2 x_1}
    +\frac{\al\, (3r-1)x_1^{r+1}}{12\log^3 x_1}-\frac{\al\, x_1^{r+1}}{4\log^4 x_1}.
\end{multline}

The unique positive root $ r_0(\al)$ of the equation
$3r^4+8r^3+6r^2-24/\al-1=0$  is decreasing on $\al$ and satisfies 
$r_0(1)= 1.1445\ldots$, $r_0(0.5)= 1.4377\ldots$ and $r_0(0.15)= 2.1086\ldots$ 
For $0.06\le r \le r_0(\al)$ and $x\ge x_1(\al)$, we have
\begin{multline}\label{Minpir1}
\pi_r(x) \ge \widehat{C_0}+
\frac{x^{r+1}}{(r+1)\log x}+\frac{x^{r+1}}{(r+1)^2\log^2 x}+
\frac{2x^{r+1}}{(r+1)^3\log^3 x}\\
  -\frac{(2\al\,r^4+7\al\, r^3+9\al\,r^2+5\al\, r+\al-6)x^{r+1}}{(r+1)^4 \log^4 x}
\end{multline}
while, if $ r > r_0(\al)$ and $x\ge x_1(\al)$, we have
\begin{multline}\label{Minpir2}
\pi_r(x) \ge \widehat{C_0}+
\frac{x^{r+1}}{(r+1)\log x}+\frac{x^{r+1}}{(r+1)^2\log^2\! x}+
\frac{2x^{r+1}}{(r+1)^3\log^3\! x}\\
-\frac{(51\al\,r^4+176\al\,r^3+222\al\,r^2+120\al\,r+23 \al\,-168)x^{r+1}}{24(r+1)^4 \log^4 x},
\end{multline}
with 
\begin{multline}
 \label{C0min}
 \widehat{C_0}= \pi_r(x_1)-\frac{x_1^r\theta(x_1)}{\log x_1}+
\frac{3\al\,r^4+8\al\,r^3+6\al\,r^2-\al-24}{24} \li(x_1^{r+1})\\
-\frac{(3\al\,r^3+5\al\,r^2+\al\,r-\al-24)x_1^{r+1}}{24\log x_1}
\\-\frac{\al\,(3r^2+2r-1)x_1^{r+1}}{24\log^2 x_1}
-\frac{\al\,(3r-1)x_1^{r+1}}{12\log^3 x_1}+\frac{\al\,x_1^{r+1}}{4\log^4 x_1}.
\end{multline}
\end{prop}

\begin{coro}\label{coropi1}
For $x \ge 110\,117\,910$, we have
\begin{equation}
  \label{pi1maj}
  \pi_1(x) \le 
\frac{x^{2}}{2\log x}+\frac{x^{2}}{4\log^2 x}+
\frac{x^{2}}{4\log^3 x}+\frac{107\;x^{2}}{160 \log^4 x}
\end{equation}
and,  for $x\ge 905\, 238\, 547$,
\begin{equation}
  \label{pi1min}
  \pi_1(x) \ge 
\frac{x^{2}}{2\log x}+\frac{x^{2}}{4\log^2 x}+
\frac{x^{2}}{4\log^3 x}+\frac{3\,x^{2}}{20 \log^4 x}.
\end{equation} 
\end{coro}

\begin{proof}
 It is Corollary 2.7 of \cite{DNh3}, cf. also \cite[Theorem 6.7 and
 Proposition 6.9]{axl}.
\end{proof}

\begin{coro}\label{coropi2}
For $x \ge 60\ 173$, 
\begin{equation}
  \label{pi2maj}
  \pi_2(x) \le 
\frac{x^{3}}{3\log x}+\frac{x^{3}}{9\log^2 x}+
\frac{2x^{3}}{27\log^3 x}+\frac{1181x^{3}}{648 \log^4 x}
\end{equation}
and for $x \ge 60\ 297$,
\begin{equation}
  \label{pi2majred}
  \pi_2(x) \le 
\frac{x^{3}}{3\log x}\left(1+\frac{0.385}{\log x}\right)
\end{equation}
while,  for $x\ge 1\ 091\ 239$, we have
\begin{equation}
  \label{pi2min}
  \pi_2(x) \ge 
\frac{x^{3}}{3\log x}+\frac{x^{3}}{9\log^2 x}+
\frac{2x^{3}}{27\log^3 x}-\frac{1069x^{3}}{648 \log^4 x}.
\end{equation} 
and for $x > 32\ 321$, with $\pi_2^-(x)=\sum_{p<x}p^2$,
\begin{equation}
  \label{pi2minred}
  \pi_2(x) \ge\pi_2^-(x) \ge 
\frac{x^{3}}{3\log x}\left(1+\frac{0.248}{\log x}\right).
\end{equation}
\end{coro}

\begin{proof}
From \eqref{dusart3}, the hypothesis $|\theta(x)-x| \le \al x/\log^3 x$
is satisfied with $\al=1$ and $x_1=89\ 967\ 803$. By computation, we  find
$\pi_2(x_1) = 13\ 501\ 147\ 086\ 873\ 627 \\ \ 946\ 348$, 
$\theta(x_1)=89\ 953\ 175.416\ 013\ 726\ldots$
and $C_0$, defined by \eqref{C0maj} with $r=2$  and $\al=1$ is equal to
$-1.040\ldots\times 10^{18} < 0$ so that \eqref{pi2maj} follows 
from \eqref{Majpir} for $x\ge x_1$.
From \eqref{fab}, the right-hand side of \eqref{pi2maj} is increasing on $x$
for $x \ge e^{4/3}=3.79\ldots$ We check that \eqref{pi2maj} holds 
when $x$ runs over the primes $p$ satisfying $60209\le p \le x_1$
but not for $p=60169$. For $x=60172.903\ldots$, the right-hand side
of \eqref{pi2maj} is equal to $\pi_2(60169)$, which completes the proof
of \eqref{pi2maj}.

Let us set $x_2=315\ 011$. From \eqref{pi2maj}, for $x\ge x_2$, we have
\begin{multline*}
  \pi_2(x) \le  \frac{x^{3}}{3\log x}\left(1+\frac{1}{\log x}
\left(\frac{1}{3}+\frac{2}{9\log x_2}+
  \frac{1181}{216\log^2 x_2}\right)\right)
\\ \le  \frac{x^{3}}{3\log x}\left(1+\frac{0.385}{\log x}\right)
\end{multline*}
which proves \eqref{pi2majred} for $x\ge x_2$.
Further, we check that
\eqref{pi2majred} holds when $x$ runs over the primes $p$ such that
$60317\le p \le x_2$ but does not hold for $p=60293$.  Solving the
equation $\pi_2(60293)=t^3/(3\log t)(1+0.385/\log t)$ yields
$t=60296.565\ldots$ which completes the
proof of \eqref{pi2majred}.

Similarly, $\widehat{C_0}$ defined by \eqref{C0min} is equal to
$8.022\ldots \times 10^{18} >0$ which implies \eqref{pi2min} from
\eqref{Minpir2} for $x\ge x_1$. Let us define
$F(t)=\frac{t^{3}}{3\log t}+\frac{t^{3}}{9\log^2 t}+
\frac{2t^{3}}{27\log^3 t}-\frac{1069t^{3}}{648 \log^4 t}$.
We have $F'(t)=\frac{t^2}{648\log^5 t}(648\log^4 t-3351\log t+4276)$
which is positive for $t > 1$ and thus, $F(t)$ is increasing for $t >
1$. For all primes $p$ satisfying $1\,091\,239 \le p \le x_1$, we
denote by $p^+$ the prime following $p$ and we check that $\pi_2(p)
\ge F(p^+)$, which proves \eqref{pi2min}.

Let us set $f(t)=\frac{t^3}{3\log t} (1+\frac{0.248}{\log
    t})$, so that $F(t)-f(t)=\frac{t^3}{81000\log^4
  t}(2304\log^2t\\ +6000 \log  t-133625)$.
The largest root of the trinomial on $\log t$  is $6.424\ldots$ so that 
$F(t) > f(t)$ holds for $t\ge 618 > \exp(6.425)$, which,
as \eqref{pi2min} holds for $x\ge 1\ 091\ 239$,
proves \eqref{pi2minred} for $x\ge 1\,091\,239$.

After that, we check that $\pi_2(p) \ge f(p^+)$ for
all pairs $(p,p^+)$ of consecutive primes satisfying
$32321 \le p<p^+ \le 1\ 091\ 239$, which,  for $x\ge 32321$, proves that 
$\pi_2(x) \ge f(x)$ and \eqref{pi2minred} if $x$ is not prime. 
For $x$ prime and $x > 32321$, we have $x-1> 32321$ and
we consider $y$ satisfying $x-1<y<x$. We have $\pi_2^-(x)=\pi_2(y)\ge
f(y)$, which proves \eqref{pi2minred} when $y$ tends to $x$.
\end{proof}

\begin{coro}\label{coropi345}
We have
\begin{equation}
  \label{majpi3}
\pi_3(x) \le 0.271\;\frac{x^4}{\log x}\quad \text{ for } \;\; x\ge 664,
\end{equation}
\begin{equation}
  \label{majpi4}
\pi_4(x) \le 0.237\;\frac{x^5}{\log x}\quad \text{ for } \;\; x\ge 200,
\end{equation}
\begin{equation}
  \label{majpi5}
\pi_5(x) \le 0.226\;\frac{x^6}{\log x}\quad \text{ for } \;\; x\ge 44
\end{equation}
\begin{equation}
  \label{majpi6}
  \pi_r(x) \le \frac{\log 3}{3}\left(1+\left(\frac 23\right)^r\right)\frac{x^{r+1}}
  {\log x}\quad \text{ for  $ \;\; x > 1$ and $r\ge 5$}.
\end{equation}
\end{coro}

\begin{proof}
First, from \eqref{C0maj}, 
with $r\in \{3,4,5\}$, $\al=1$ and $x_1=89\,967\,803$, we calculate 
\[C_0(3)=-1.165\ldots 10^{26},\quad C_0(4)=-1.171\ldots 10^{34},  
\quad C_0(5)=-1.123\ldots 10^{42}.\]
As these three numbers are negative, from \eqref{Majpir}, 
for $r\in \{3,4,5\}$ and $x \ge x_1$, we have
\begin{multline}\label{pir<}
\pi_r(x) \le 
\frac{x^{r+1}}{(r+1)\log x}\Big(1+\frac{1}{(r+1)\log  x_1}+
\frac{2}{(r+1)^2\log^2\! x_1} \\
+ \frac{51r^4+176r^3+222r^2+120r+191}{24(r+1)^3\log^3x_1}\Big)
\end{multline}
and
\[\pi_3(x) \le 0.254\;\frac{x^4}{\log x},\quad
\pi_4(x) \le 0.203\;\frac{x^5}{\log x}\qtx{ and }
\pi_5(x) \le 0.169\;\frac{x^6}{\log x}.\]
If $p$ and $p^+$ are two consecutive primes, for $r\ge 3$, it follows
from \eqref{fab} that the function 
$t\mapsto \frac{\pi_r(t)\log t}{t^{r+1}}$ is decreasing on $t$ for
$p\le t < p^+$. Therefore, to 
complete the proof of \eqref{majpi3}, one checks that
$\frac{\pi_3(p)\log p}{p^4}\le 0.271$ holds for $673\le p \le x_1$
and we solve the equation $\pi_3(t)\log
t=0.271\;t^4$ on the interval $[661,673)$ whose root is
$663.35\ldots$
The proof of \eqref{majpi4} and \eqref{majpi5} are similar.

Since the mapping $t\mapsto (\log t)/t^6$ is decreasing for $t\ge 2$,
the maximum of $\pi_5(x)(\log x)/x^6$ is attained on a prime $p$. In
view of \eqref{majpi5}, calculating $\pi_5(p)(\log p)/p^6$ 
for $p=5,7,\ldots,43$ shows that the maximum for $x\ge 5$ is attained
for $x=5$ and is equal to $0.350\ldots < (\log 3)/3$.
Let $r$ be a number $\ge 5$. By applying the trivial inequality
\[\pi_r(x)\le x^{r-5}\pi_5(x),\quad r\ge 5,\]
we deduce that
 $\pi_r(x) < (\log 3)\,x^{r+1}/(3\log x)$ for  $x\ge 5,$
 and, by calculating $\pi_r(2)(\log 2)/2^{r+1}=(\log 2)/2$ and 
$\pi_r(3)(\log 3)/3^{r+1}=\left(1+\left(\frac 23\right)^r\right)\frac{\log 3}{3}$, 
we obtain \eqref{majpi6}.
\end{proof}

\section{$\ell$-superchampion numbers}\label{parSupCh}

\subsection{Definition of $\ell$-superchampion numbers}\label{pardefSupCh}

\begin{definition}\label{defSupCh}
An integer $N$ is said $\ell$-superchampion (or more simply
superchampion) if there exists $\rho >0$ such that, for all integer $M \ge 1$
\begin{equation}
  \label{SupCh}
  \ell(M)-\rho \log M \ge \ell(N)-\rho \log N.
\end{equation}
When this is the case, we say that $N$ is a $\ell$-superchampion {\em
  associated} to $\rho$.
\end{definition}

Geometrically, if we represent $\log M$ in abscissa and $\ell(M)$ in
ordinate for all $M\ge 1$, the vertices of the convex envelop of all
these points represent the $\ell$-superchampion numbers
(cf. \cite[Fig. 1, p. 633]{DNZ}).
If $N$ is an $\ell$-superchampion, the following property holds
(cf. \cite[Lemma 3]{DNZ}):
\begin{equation}\label{NgN}
N=g(\ell(N)).
\end{equation}

Similar numbers, the so-called {\it superior highly composite numbers} were 
first introduced by S. Ramanujan (cf. \cite{Ram}). The $\ell$-superchampion
numbers were also used in
\cite{TheseJLN,CRASJLN,Mas,MNRAA,MNRMC,MPE,DN}.
Let us recall the properties we will need.
For more details, 
cf. \cite[Section 4]{DNZ}. 


\begin{lem}
  Let $\rho$ satisfy $\rho \ge 5/\log 5 = 3.11\ldots$. Then,
  depending on $\rho$, there exists an unique decreasing sequence $ (\xi_j)_{j \ge 1}$ such that
$\xi_1 > \exp(1)$ and, for all $j \ge 2$, 
\begin{equation}\label{xik}
  \xi_j > 1 \qtx{ and }\frac{\xi_j^j - \xi_j^{j-1}}{\log \xi_j} = \frac{\xi_1}{\log \xi_1} = \rho.
\end{equation}
We have also $\xi = \xi_1 \ge 5$ and $\xi_2 \ge 2$.
 \end{lem}

\suppress{For every $\rho > 0$ there exitst at least one superchampion
associated to $\rho$. For some values of $\rho$ there are more
than one superchampion. These values of $\rho$ are the
elements of the set $\ce$ defined below.
}

\begin{definition}
For each prime $p\in \cp$, let us define the sets
\begin{equation}\label{Ep}
\ce_p =\left\{\frac{p}{\log p},\frac{p^2-p}{\log p}\,,\ldots,
\frac{p^{i+1}-p^i}{\log p}\,,\ldots \right\},
\quad
\ce= \bigcup_{p\in \cp} \ce_p.
\end{equation}
\end{definition}

\begin{rem}
Note that all the elements of $\ce_p$ are distinct at the exception, for $p=2$,
of $\dfrac{2}{\log 2}=\dfrac{2^2-2}{\log 2}$ and that, for 
$p \ne q$, $\ce_p\cap \ce_q=\emptyset$ holds.

Furthermore if $\rho \in \ce_p$, there is an unique $\jh=\jh(\rho) \ge 1$ such that
   $\xi_{\jh}$ is an integer; this integer is $p$ and
   $\jh$ is given by
   \begin{equation}\label{jhdef}
     \jh =
     \begin{cases}
       1 &\textit{ if }  \rho = p/\log p\\
       j  &\textit{ if }  \rho = (p^j-p^{j-1})/\log p \quad ( j \ge 2).
     \end{cases}
   \end{equation}
 \end{rem}
 
\begin{prop}\label{propNrho}
  Le $\rho > 5 / \log 5$, $\xi_j = \xi_j(\rho)$ defined by
  \eqref{xik} and $\Nro$, $\Nrop$  defined by
  \begin{equation}\label{Nrho}
    \Nro = \prod_{j\ge 1}\big(\prod_{p < \xi_j} p\big)
    =\prod_{j \ge 1} \big(\prod_{\xi_{j+1} \le p < \xi_j} p^j\big)
  \end{equation}
  and
 \begin{equation}\label{Nrhop}
   \Nrop = \prod_{j\ge 1}\big(\prod_{p \le\xi_j} p\big)
   =\prod_{j \ge 1} \big(\prod_{\xi_{j+1} <  p \le  \xi_j} p^j \big)
 \end{equation}
 Then,
 \begin{enumerate}
 \item
   If $\rho \notin \ce$, $\Nro = \Nrop$ is the unique superchampion
   associated to $\rho$.
 \item
   If $\rho \in \ce_p$, $\Nro$ and $\Nrop$ are two consecutive
   superchampions, they are the only superchampions
   associated to $\rho$,
   and
   \begin{equation}\label{majNrop}
     \Nrop = p \Nro = \xi_{\jh}(\rho) \Nro.
    \end{equation}
  \end{enumerate}
\end{prop}

From \eqref{xik} we deduce that the upper bound for $j$
in  \eqref{Nrho} and \eqref{Nrhop} 
is $\lfloor J\rfloor$ with $J$ defined by
\begin{equation}
  \label{J}
\frac{2^J-2^{J-1}}{\log 2}=\rho=\frac{\xi}{\log \xi} \qtx{i.e.}
J=\frac{\log \xi+\log(2\log 2) -\log \log \xi}{\log 2} 
< \frac{\log \xi}{\log 2}, 
\end{equation}
as $\xi \ge 5$ is assumed.

  \begin{definition}\label{nsuperch}
    Let us suppose $n\ge 7$.
    Depending on $n$, we define $\rho$, $N'$, $N''$, $n'$, $n''$, $\xi$,
  and $(\xi_j)_{j \ge 1}$.
  \begin{enumerate}
  \item
    $\rho$ is the unique element of $\ce$ such that
    \begin{equation}\label{defrhon}
      \ell(N_{\rho})  \le n < \ell(N_{\rho}^+).
    \end{equation}
  \item
      $N', N'', n',n''$ are defined by
      \begin{equation}\label{defN'n'}
        N' = \Nro, \quad N''=\Nrop
        \quad n' = \ell(N'),\text{ and } n''=\ell(N''). 
      \end{equation}
    \item
      For $j \ge 1$,  $\xi_j$ is defined by \eqref{xik}
      and $\xi$ is defined by $\xi = \xi_1$ i.e. $\log\xi/\xi = \rho.$
    \end{enumerate}    
\end{definition}

\begin{prop}
 Let us suppose $n \ge 7$, $\rho$, $N'$, $N''$, $n'$, $n''$ and $\xi $
 defined by Definiton \ref{nsuperch}. Then
  \begin{equation}\label{N'n}
      n' \le n <  n''   \qtx{and}  N' =  g(n') \le g(n)  < N'' = g(n'') = p N'. 
    \end{equation}
    \begin{equation}\label{ellN'N''}
  \ell(N')-\rho \log N'=\ell(N'')-\rho \log N''.
\end{equation}
    \begin{equation}\label{N''xiN'}
      N'' \le \xi N'.
    \end{equation}
    \begin{equation}\label{ellN'N''xi}
      \ell(N'')-\ell(N')  \le \xi.
    \end{equation}
\end{prop}

\begin{proof}
  \begin{itemize}
  \item
    \eqref{N'n} results of \eqref{defrhon}, \eqref{defN'n'} and \eqref{majNrop}.

  \item
    By applying \eqref{SupCh} first to $M=N'$ , $N=N''$ and further to
    $M=N''$, $N=N'$ we get \eqref{ellN'N''}.

  \item
    Equality \eqref{majNrop} gives $N'' = \xi_{\jh(\rho)} N'$;  with the
    decreasingness of $(\xi_j)$ and the definiton $\xi = \xi_1$ we
    get \eqref{N''xiN'}.
  \item
    By using \eqref{ellN'N''} and \eqref{N''xiN'}  we have
    $\ell(N'')-\ell(N')=\rho \log (N"/N')) \le \rho \log \xi=\xi$.
  \end{itemize}
\end{proof}

In the array of Fig. \eqref{arraygn}, for a small $n$, one can read the value of
$N',N'',\rho$ and $\xi$ as  given in Definition \ref{nsuperch}. For
instance, for $n=45$, we have $N'=60\,060$, $N''=180\,180$,
$\rho=6/\log 3$ and $\xi=14.667\ldots$ We also can see the values of
the parameter associated to a superchampion number $N$. For instance,
$N= 360\,360$ is associated to all values of $\rho$ satisfying $4/\log
2 \le \rho \le 17/\log 17$.

\begin{figure}[!ht]
    \[
  \begin{array}{|c|c|c|c|c|}
\hline
n&N&\ell(N)&\rho&\xi\\
\hline
  &12 = 2^2\cdot 3&7&  &  \\
 7 .. 11 &  &  & 5/\log 5=3.11& 5\\
   &60 = 2^2\cdot 3\cdot 5&12&  &  \\
 12  ..   18 &  &  & 7/\log 7=3.60& 7\\
   &420 = 2^2\cdot 3\cdot 5\cdot7&19&  &  \\
 19 ..  29&  &  & 11/\log 11=4.59& 11\\
   &4\,620 = 2^2\cdot 3\cdot 5\cdot 7 \cdot 11&30&  &  \\
30 ..  42&  &  & 13/\log 13=5.07& 13\\
   &60\,060 = 2^2\cdot 3\ldots 13&43&  &  \\
 43 ..  48&  &  & (9-3)/\log 3=5.46&14.66\\
    &180\,180 = 2^2\cdot 3^2\cdot 5\ldots 13&49&  &  \\
 49 ..  52&  &  & (8-4)/\log 2=5.77&16\\
    &360\,360 = 2^3\cdot 3^2\cdot 5\ldots 13&53&  &  \\
 53 ..  69&  &  & 17/\log 17=6.00&17\\
    &6\,126\,120 = 2^3\cdot 3^2\cdot 5\ldots 17&70&  &  \\
 70..88&  &  & 19/\log 19=6.45&19\\
    &116\,396\,280 = 2^3\cdot 3^2\cdot 5\ldots 19&89&  &  \\
\hline
  \end{array}
\]
\caption{The first  superchampion numbers}
\label{arraygn}
\end{figure}

As another example, let us consider $x_0=10^{10}+19$, the smallest prime
exceeding $10^{10}$, and the two $\ell$-superchampion numbers 
$N'_0$ and $N''_0$ associated to $\rho =x_0/\log x_0 \in \ce_{x_0}$. We have
\begin{multline}
   \label{N'0}
N'_0= 2^{29}3^{18}5^{12}7^{10}11^813^817^719^723^6\ldots
31^637^5\ldots71^573^4\ldots 211^4 223^3\ldots 1459^3\\ 
1471^2\ldots 69557^2 69593\ldots
    9\ 999\ 999\ 967 \qtx{and} N''_0=x_0\,N'_0,  
\end{multline}
\begin{equation}
  \label{n0}
  \nu_0=\ell(N'_0)
=2\ 220\ 832\ 950\ 051\ 364\ 840=2.22\ldots 10^{18}, \qquad J=29.165\ldots
\end{equation}
and, for $2\le j \le 29$, $\xi_j=x_j^{(0)}$ with
\begin{multline}\label{xj0}
x_2^{(0)}=69588.8\ldots, x_3^{(0)}=1468.8\ldots, 
x_4^{(0)}=220.2\ldots, x_5^{(0)}=71.5\ldots \\
\text{ and }\;   x_{29}^{(0)}=2.0\ldots 
\end{multline}
A complete table of values of $x_j^{(0)}$ is given in \cite{web}.

Let $n$ be an integer, and $\xi = \xi(n)$ (Definition \ref{nsuperch}). Let us suppose $n \ge
  \nu_0 = \ell(N_0')$; then, by \eqref{defrhon}, $\rho \ge \rho_0$, ie
  $\xi / \log \xi  \ge x_0 / \log x_0$. So that
  \begin{equation}\label{ngreatern0}
    \xi \ge x_0.
  \end{equation}

\subsection{Estimates of $\xi_j$ defined by \eqref{xik}}\label{parxij}

\begin{lem}\label{lemxik}
(i) For $\xi\ge 5$ and $j\ge 2$, we have
\begin{equation}
  \label{xk1sk}
\xi_j \le \xi^{1/j}. 
\end{equation}
(ii) For $2\le j\le 8$ and $\xi \ge \la_j$, we have
\begin{equation}
  \label{xsk1sk}
\xi_j \le \left(\frac{\xi}{j}\right)^{1/j}, 
\end{equation}
with $\la_2=80$, $\la_3=586$, $\la_4=6381$, $\la_5=89017$, 
$\la_6=1\,499\,750$, $\la_7=29\,511\,244$, $\la_8=663\,184\,075$.

(iii) For $\xi \ge 5$ and $j$ such that $\xi_j\ge x_j^{(0)} \ge2$
(where $x_j^{(0)}$ is defined in \eqref{xj0}), we have
\begin{equation}
  \label{xkksk}
\xi_j \le \left(\frac{\xi}{j(1-1/x_j^{(0)})}\right)^{1/j} \le 
 \left(\frac{2\xi}{j}\right)^{1/j}.
\end{equation}
\end{lem}

\begin{proof}
(i) As the function $t\mapsto \frac{t^j-t^{j-1}}{\log t}$ is increasing, it suffices to show that
\[\frac{(\xi^{1/j})^j-(\xi^{1/j})^{j-1}}{\log(\xi^{1/j})}=
\frac{\xi-\xi^{(j-1)/j}}{(1/j)\log \xi} \ge \frac{\xi}{\log \xi}\]
which is equivalent to
\begin{equation}\label{xj1}
\xi \ge \left(1+\frac{1}{j-1}\right)^j.
\end{equation}
But the sequence $(1+1/(j-1))^j$ decreases from 4 to $\exp(1)$ when
$j$ increases from $2$ to $\iy$, and $\xi \ge 5$ is assumed, which
implies \eqref{xj1} and \eqref{xk1sk}.

(ii)  Here, we have to prove
\[\frac{\xi/j-(\xi/j)^{(j-1)/j}}{(1/j)\log (\xi/j)} \ge \frac{\xi}{\log \xi}\]
which is equivalent to
\begin{equation}
  \label{xij}
\frac{\xi^{1/j}}{\log \xi} \ge\frac{j^{1/j}}{\log j}. 
\end{equation}
For $2\le j\le 8$, we have $\la_j > e^j$ and, from \eqref{fab}, the function 
$\xi\mapsto \xi^{1/j}/\log \xi$ is increasing for $\xi\ge \la_j$
and its value for $\xi=\la_j$ exceeds $j^{1/j}/\log j$ so that
inequality \eqref{xij} is satisfied.

(iii) Let us suppose that $\xi\ge 5$ and  $\xi_j\ge x_j^{(0)}\ge 2$
hold. From \eqref{xik} and \eqref{xk1sk}, we have 
\[\xi_j^j=\frac{\xi \log \xi_j}{\log \xi (1-1/\xi_j)}\le 
\frac{\xi}{j (1-1/\xi_j)}\le \frac{\xi}{j(1-1/x_j^{(0)})}\le \frac{2\xi}{j}\]
which proves \eqref{xkksk}.
\end{proof}

\begin{coro}\label{coropjxi}
For $n\ge 7$,  $\rho = \rho(n)$, $\xi$=$\xi(n)$ defined in Definition
\ref{nsuperch}, the powers $p^j$ of primes dividing $N'=\Nro$ or $N''=
\Nrop$ in \eqref{Nrho} or \eqref{Nrhop} do not exceed $\xi$.
\end{coro}

\begin{proof}
  This follows from  \eqref{xk1sk}, \eqref{Nrho} and \eqref{Nrhop}.
\end{proof}

\subsection{Convexity}\label{parcon}
In this paragraph we prove Lemma \ref{lemanzn} which is used
to speed-up the computations done in 
Section \ref{parproof174} to prove Inequality \eqref{maxgn}
of Theorem \ref{thmhgMm},
and   and in Section \ref{parproofgn4}
to prove Theorem \ref{thmgnHR} (iv).

\begin{lem}\label{lemPhicon}
The function $t\mapsto \sqrt{\limm(t)}$ is concave for 
$t  > \li(e^2)=4.954\ldots$.

Let $a\le 1$ be a real number. Then, for $t \ge 31$, the function 
$t\mapsto \sqrt{\limm(t)}-a(t(\log t))^{1/4}$ is concave.  
\end{lem}

\begin{proof}
The proof is a good exercise of calculus: cf. \cite[Lemma 2.5]{DNh3}.
\end{proof}

\begin{lem}\label{lemconz}
Let $u$ be a real number, $0 \le u \le e$. The function $\Phi_u$
defined by
\begin{equation}\label{PHIu}
\Phi_u(t) =\sqrt{t\log t}\left(1+\frac{\log\log t-1}{2 \log t}-
u\frac{(\log \log t)^2}{\log^2 t}\right)
\end{equation}
is increasing and concave for $t\ge e^e=15.15\ldots$  
\end{lem}

\begin{proof}
Let us write $L$ for $\log t$ and $\la$ for $\log \log t$. 
One calculates (cf. \cite{web})
\[\frac{d\,\Phi_u}{dt}=\frac{1}{4L^2\sqrt{tL}}\left(2L(L^2-u\la^2)+L(L-\la+3)
+\la(L^2-2u)+6u\la(\la-1)\right)\]
\begin{multline*}
\frac{d^2\Phi_u}{dt^2}=\frac{-1}{8L^2(tL)^{3/2}}(
L^2(L^2-2u\la^2)+L^3(\la-1)+L(2L-3\la)+(L^4+11L-22u\la)\\
+2u(15\la^2-21\la+8)).
\end{multline*}
For $t\ge e^e$, we have $L\ge e$, $\la\ge 1$, $L=e^\la> e\la \ge u\la$, 
so that $\Phi'_u$ is positive. 

In $\Phi''_u$, $L^4+11L\ge (e^3+11)L\ge (e^3+11)u\la>22u\la$,
the trinomial $15\la^2-21\la+8$ is always positive
and the three first terms of the parenthesis are also positive, 
so that $\Phi''_u$ is negative.
\end{proof}

\begin{lem}\label{lemgncon}
Let $n'=\ell(N')$, $n''=\ell(N'')$ where $N'$ and $N''$ are two
consecutive $\ell$-superchampion numbers
associated to the same parameter $\rho$.
Let $\Phi$ be a concave function on
$[n',n'']$ such that $\log N'\le \Phi(n')$ and $\log N'' \le
\Phi(n'')$. 
Then,
\[\text{For } n \in [n', n''],\quad \log g(n) \le \Phi(n).\]
\end{lem}

\begin{proof}
From \eqref{NgN}, it follows that $N'=g(n')$ and $N''=g(n'')$.
Let us set $N=g(n)$ so that, from \eqref{ellh}, we have $n\ge \ell(N)$. 
From the definition \eqref{SupCh} of superchampion numbers and
\eqref{ellN'N''}, we have
\begin{equation}
  \label{nxiN}
n-\rho\log N\ge \ell(N)-\rho\log N\ge n'-\rho\log N'=n''-\rho\log N''.
\end{equation}
Now, we may write
\begin{equation}
  \label{malbe}
  n=\al n'+\beta n''\quad \text{ with }\quad 0\le \al\le 1 
\qtx{and}\beta=1-\al
\end{equation}
and \eqref{nxiN} implies
\begin{multline}
  \label{albexi}
  \log N \le \frac 1\rho(n-(n'-\rho\log N'))
 \\  = \frac 1\rho (\al n' +\beta n''-\al(n'-\rho \log
  N')-\beta(n''-\rho\log n''))\\
=\al \log N'+\beta \log N''.
\end{multline}
From the concavity of $\Phi$,  
$\log N'\le \Phi(n')$ and  $\log N''\le \Phi(n'')$,
\eqref{albexi} and \eqref{malbe} imply
\[\log g(n)=\log N\le \al\ \Phi(n')+\beta\  \Phi(n'')\le 
\Phi(\al n'+\beta n'')=\Phi(n),\]
which completes the proof of Lemma \ref{lemgncon}.
\end{proof}

For $n\ge 2$, let us define $z_n$ by
\begin{equation}\label{zn}
\log g(n)=\sqrt{n\log n}\left(1+\frac{\log \log n-1}{2\log
    n}-z_n\frac{(\log \log n)^2}{\log^2 n}\right).
\end{equation}

\begin{lem}\label{lemanzn}
Let $N'$ and $N''$ be two consecutive $\ell$-superchampion numbers and
$\ell(N') \le n \le \ell(N'')$.

 (i) If $n'\ge 43$ and if $a_{n'}$ and $a_{n''}$ (defined by \eqref{an})
 both belong to $[0,1]$ then
\[a_n\ge \min(a_{n'},a_{n''}).\]

 (ii) If $n'\ge 19$ and if $z_{n'}$ and $z_{n''}$ (defined by \eqref{zn})
 both belong to $[0,e]$ then
\[z_n\ge \min(z_{n'},z_{n''}).\]
\end{lem}

\begin{proof}
From \eqref{NgN}, it follows that $N'=g(n')$ and $N''=g(n'')$.
Let us set $N=g(n)$ and
\[\Phi(t)=\sqrt{\limm(t)}-\min(a_{n'},a_{n''})(t\log t)^{1/4}.\]
From Lemma \ref{lemPhicon}, $\Phi$ is concave on $[n',n'']$.
Moreover, from the definition \eqref{an} of $a_{n'}$ and $a_{n''}$, we have 
$\log N'\le \Phi(n')$ and  $\log N''\le \Phi(n'')$ 
which, from Lemma \ref{lemgncon}, implies
$\log g(n)\le \Phi(n)$
and, from \eqref{an}, $a_n\ge \min(a_{n'},a_{n''})$ holds,
which proves (i).

The proof of (ii) is similar. We set $u=\min(z_{n'},z_{n''})$. From Lemma
\ref{lemconz}, $\Phi_u$ is concave on $[n',n'']$, $\log g(n')\le \Phi_u(n')$ 
and $\log g(n'')\le \Phi_u(n'')$ so that, from Lemma \ref{lemgncon},
we have $\log g(n) \le \Phi_u(n)$  and, from \eqref{zn}, $z_n\ge u$
holds, which completes the proof of Lemma \ref{lemanzn}.
\end{proof}

\subsection{Estimates of $\xi_2$ defined by \eqref{xik}}\label{parxi2}
By iterating the formula 
$\xi_2=\sqrt{\frac{\xi\log \xi_2}{\log \xi} +\xi_2}$
(cf. \eqref{xik}), for any positive integer $K$, we get
\begin{equation}
  \label{x2asy}
\xi_2=\sqrt{\frac{\xi}{2}}\left(1+\sum_{k=1}^{K-1} \frac{\al_k}{\log ^k \xi}+
\co\left(\frac{1}{\log^K \xi}\right)\right),\quad \xi\to\iy,
\end{equation}
with
\[
  \al_1=-\frac{\log 2}{2},\quad \al_2=-\frac{(\log 2)(\log 2 +4)}{8},
  \quad \al_3=-\frac{(\log 2)(\log^2 2 + 8 \log 2+8)}{16}
\]

\begin{prop}\label{propx2}
  We have the following bounds for $\xi_2$ :
  \begin{align}\label{x2345}
    \xi_2 &<  \sqrt{\frac{\xi}{2}}\left(1-\frac{\log 2}{2\log
    \xi}\right)
    <\sqrt{\frac{\xi}{2}}\left(1-\frac{0.346}{\log \xi}\right)
    &\textrm{ for  } &\xi \ge 31643  \\
      \label{x2366}
  \xi_2 &>  \sqrt{\frac{\xi}{2}}\left(1-\frac{0.366}{\log \xi}\right)
  &\textrm{ for } &\xi \ge 4.28\times 10^9.
\end{align}
\end{prop}

\begin{proof}
Let us suppose $\xi \ge 4$ and $a\le 0.4$. We set
\[\Phi=\Phi_a(\xi)=
\sqrt{\frac{\xi}{2}}\left(1-\frac{a}{\log \xi}\right) \ge 
\sqrt{\frac{4}{2}}\left(1-\frac{0.4}{\log 4}\right)=
1.006\ldots > 1\]
and
\begin{multline*}
  W=W_a(\xi)  =\frac{(\Phi^2-\Phi)\log \xi}{\xi}-\log \Phi
=\frac{(\log \Phi)(\log \xi)}{\xi}
\left(\frac{\Phi^2-\Phi}{\log \Phi}-\frac{\xi}{\log \xi}\right)\\
=\frac{(\log \Phi)(\log \xi)}{\xi}
\left(\frac{\Phi^2-\Phi}{\log \Phi}-\frac{\xi_2^2-\xi_2}{\log \xi_2}\right).
\end{multline*}
As $\Phi> 1$ and $t\mapsto (t^2-t)/\log t$ is increasing, we have
\begin{equation}
  \label{x2W>0}
W_a(\xi) > 0 \quad \Llr \quad \xi_2 < \Phi(\xi)=
\sqrt{\frac{\xi}{2}}\left(1-\frac{a}{\log \xi}\right).
\end{equation}
By the change of variable
\[\xi=\exp(2t), \quad t=\frac 12\log \xi\ge \log 2,\]
with the help of Maple (cf. \cite{web}), we get
\begin{align*}
 &W =-a+\frac{a^2}{4t}+\frac 32 \log 2+\log t-\log(2t-a)-
e^{-t}\frac{\sqrt 2}{2}(2t-a),\\
&W' =\frac{dW}{dt}=\frac{U}{4t^2(2t-a)} \qtx{with}\\
&U   =-2a(a+2)t+a^3+2\sqrt 2 e^{-t} V (4t^4-4(a+1)t^3+a(a+2)t^2),\\
  &U'  =\frac{dU}{dt}\\
  &=-2a(a+2)+2\sqrt 2 e^{-t}(-4t^4+ 4(a+5)t^3-(a^2+14a+12)t^2+2a(a+2)t)\\
& U'' =\frac{d^2U}{dt^2}=2\sqrt 2 e^{-t} V\\
 &      \textrm{ with }    V=4t^4-4(a+9)t^3+(a^2+26a+72)t^2-4(a^2+8a+6)t+2a(a+2).
 \end{align*}       
The sign of $U''$ is the same than the sign of the polynomial $V$ that
is easy to find. For $a$ fixed and $t\ge \log 2$ (i.e. $\xi\ge 4)$,
one can successively determine, the variation and the sign of $U'$,
$U$ and $W$.

\paragraph{For $a=(\log 2)/2=0.346\ldots$} 
\begin{equation*}
 \begin{array}{c|ccccccccccc|}
 t& \log 2& &0.96& &2.39& &4.23& &6.38& &\iy\\
\hline
  U''& &+ & &+&0&-& &-&0&+&  \\
\hline
  & & & & & 16.09& & & & & &-1.62\\
U'&  &\nearrow& 0  &\nearrow& &\searrow& 0  &\searrow&
   & \nearrow & \\
  & -2.79 & & & & & & & &-10.03& & \\
\hline
\end{array}
\end{equation*}
\begin{equation*}
 \begin{array}{c|ccccccccccc|}
 t& \log 2& &0.96& &1.49& &4.23& &8.0& &\iy\\
\hline
  U'& &- &0 &+& &+&0&-& &-&  \\
\hline
  &-1.76& & & & & &29.6 & & & & \\
U&  &\searrow&   &\nearrow&0 &\nearrow&   &\searrow& 0 & \searrow & \\
  &  &  &-2.18 & & & & & & & &-\iy \\
\hline
\end{array}
\end{equation*}
\begin{equation}\label{tab1W}
 \begin{array}{c|ccccccccc|}
 t& \log 2& &1.49& &5.18& &8.0& &\iy\\
\hline
  U \text{ or } W'&  &- &0 &+& & +& 0&-&  \\
\hline
  &-0.036& & & & & &0.021 & & \\
W&  &\searrow&   &\nearrow&0 &\nearrow& &\searrow&  \\
  &  & &-0.27& & & & & &0 \\
\hline
\end{array}
\end{equation}
The root of $W$ is $w_0=5.1811243\ldots$ and $\exp(2w_0)=31642.25\ldots$.
Therefore, \eqref{x2345} follows from array \eqref{tab1W}.

\paragraph{For $a=0.366$} 
\begin{equation*}
  \begin{array}{c|ccccccccccc|}
\hline
 t& \log 2& &0.97& &2.39& &4.23& &6.39& &\iy\\
\hline
  U''& &+ & &+&0&-& &-&0&+&  \\
\hline
  & & & & & 15.91& & & & & &-1.731\\
U'&  &\nearrow& 0  &\nearrow& &\searrow& 0  &\searrow&
   & \nearrow & \\
  & -2.957 & & & & & & & &-10.09& & \\
\hline
\end{array}
\end{equation*}
\begin{equation*}
  \begin{array}{c|ccccccccccc|}
    \hline
 t& \log 2& &0.9797& &1.5208& &4.231& &7.892& &\iy\\
\hline
  U'& &- &0 &+& &+&0&-& &-&  \\
\hline
  &-1.830& & & & & &29.04 & & & & \\
U&  &\searrow&   &\nearrow&0 &\nearrow&   &\searrow& 0 & \searrow & \\
  &  &  &-2.308 & & & & & & & &-\iy \\
   \hline
\end{array}
\end{equation*}
\begin{equation}\label{tab2W}
  \begin{array}{c|ccccccccccc|}
    \hline 
       t& \log 2& &1.52& &6.37& &7.89& &11.08& &\iy\\
   \hline
   W'& &-&0&+& &+&0&-& &-&\\
\hline
  &-0.025& & & & & &0.004 & &  &  &\\
W&  &\searrow&   &\nearrow&0 &\nearrow& &\searrow&0 &  \searrow & \\
  &  & &-0.28 & & & & & & & &-0.019 \\
   \hline
\end{array}
\end{equation}
The root $w_0$ of $W$ is equal to $11.08803\ldots$ and
$\exp(2w_0)=4.27505\ldots \times 10^9$ so that 
\eqref{x2366} follows from \eqref{x2W>0} and from array \eqref{tab2W}.

\end{proof}
\begin{rem}\label{remxi2}
By solving the system $W=0$, $U=0$ on
the two variables $t$ and $a$, one finds
\[a=a_0=0.370612465\ldots,\qquad t=t_0=7.86682407\ldots \]
and for $a=a_0$, $t=t_0$ is a double root of $W_{a_0}$. By studying
the variation of $W_{a_0}$, we find an array close to \eqref{tab2W},
but $W_{a_0}(t_0)=W'_{a_0}(t_0)=0$, so that $W_{a_0}$
is nonpositive for $t\ge \log 2$, which proves
\[\xi_2 \ge  \sqrt{\frac{\xi}{2}}\left(1-\frac{a_0}{\log \xi}\right)
> \sqrt{\frac{\xi}{2}}\left(1-\frac{0.371}{\log \xi}\right) \qtx{for}
\xi \ge 4.\]
\end{rem}

\begin{coro}\label{coro1slogxi2}
If $\xi \ge x_0=10^{10}+19$ holds and $\xi_2$ is defined by
\eqref{xik}, then
\begin{equation}
  \label{1slogxi2}
  \frac{2}{\log \xi}\le \frac{2}{\log \xi}\left(1+\frac{\log 2}{\log
      \xi}\right) \le \frac{1}{\log \xi_2} \le \frac{2}{\log
    \xi}\left(1+\frac{0.75}{\log \xi}\right)\le \frac{2.07}{\log \xi}.
\end{equation}
\end{coro}

\begin{proof}
 Since $\xi \ge x_0$ holds and 
$\xi_2$ is increasing on $\xi$, we have 
$\xi_2\ge x_2^{(0)}=69588.859\ldots$ (cf. \eqref{xj0}) 
and \eqref{xsk1sk} implies $\xi_2 \le \sqrt{\xi/2}$, i.e.
\[\log \xi_2 \le \frac 12\log \xi-\frac 12\log 2=
\frac{\log \xi}{2}\left(1-\frac{\log 2}{\log \xi}\right)\]
and
\begin{equation*}
  \frac{1}{\log \xi_2}\ge \frac{2}{(\log \xi)(1-(\log 2)/\log \xi)}
\ge \frac{2}{\log \xi}\left(1+\frac{\log 2}{\log \xi}\right)
\ge \frac{2}{\log \xi}.
\end{equation*}
On the other hand, \eqref{x2366} implies
\begin{eqnarray*}
  \log \xi_2 &\ge& \frac 12\log \xi-\frac 12\log 2+
                   \log\left(1-\frac{0.366}{\log \xi}\right)\\
 &\ge&
\frac 12\log \xi-\frac 12\log 2+\log\left(1-\frac{0.366}{\log x_0}\right)\\
&\ge&  \frac 12\log \xi- 0.363 = 
\frac{\log \xi}{2}\left(1-\frac{0.726}{\log \xi}\right)\ge
\frac{\log \xi}{2}\left(1-\frac{0.726}{\log x_0}\right)\ge \frac{\log \xi}{2.07}
\end{eqnarray*}
and by \eqref{1s1meps},
\begin{eqnarray*}
  \frac{1}{\log \xi_2}&\le& \frac{2}{(\log \xi)(1-0.726/\log \xi)}
\le \frac{2}{\log \xi}
\left(1+\frac{0.726}{(\log \xi)(1-0.726/\log x_0)}\right)\notag\\
&\le& \frac{2}{\log \xi}\left(1+\frac{0.75}{\log \xi}\right)
\le \frac{2}{\log \xi}\left(1+\frac{0.75}{\log x_0}\right)
\le \frac{2.07}{\log \xi},
\end{eqnarray*} 
 which completes the proof of \eqref{1slogxi2}.
\end{proof}

\begin{coro}\label{corothxi2}
If $\xi \ge x_0=10^{10}+19$ holds and $\xi_2$ is defined by
\eqref{xik}, then
\begin{equation}
  \label{thxi2}
 \sqrt{\frac{\xi}{2}}\left(1-\frac{0.521}{\log \xi}\right)
\le \theta^-(\xi_2)=\sum_{p<\xi_2} \log p \le \theta(\xi_2) \le
 \sqrt{\frac{\xi}{2}}\left(1-\frac{0.346}{\log \xi}\right).
\end{equation}
\end{coro}

\begin{proof}
We have $\xi_2> x_2^{(0)}> 69588$ and \eqref{eq7461},
\eqref{1slogxi2} and \eqref{x2366} imply
  \begin{eqnarray*}
\theta^-(\xi_2) &\ge & \xi_2\left(1-\frac{0.0746}{\log \xi_2}\right)
\ge \xi_2\left(1-\frac{0.0746\times 2.07}{\log \xi}\right)
\ge \xi_2\left(1-\frac{0.155}{\log \xi}\right)\\
&\ge& \sqrt{\frac{\xi}{2}}\left(1-\frac{0.366}{\log \xi}\right) 
\left(1-\frac{0.155}{\log \xi}\right)\ge
\sqrt{\frac{\xi}{2}}\left(1-\frac{0.521}{\log \xi}\right).
\end{eqnarray*}
Similarly, for the upper bound, we use \eqref{eq79}, \eqref{1slogxi2}
and \eqref{x2345} to get
  \begin{eqnarray*}
\theta(\xi_2) &\le & \xi_2 \left(1+\frac{0.000079}{\log \xi_2}\right)
\le \xi_2\left(1+\frac{0.000079\times 2.07}{\log \xi}\right)\\
&\le& \xi_2\left(1+\frac{0.000164}{\log \xi}\right)
\le \sqrt{\frac{\xi}{2}}\left(1-\frac{\log 2}{2\log \xi}\right) 
      \left(1+\frac{0.000164}{\log \xi}\right)\\
 &\le&
\sqrt{\frac{\xi}{2}}\left(1-\frac{0.346}{\log \xi}\right)
\end{eqnarray*}
which proves the upper bound of \eqref{thxi2}
\end{proof}

\begin{coro}\label{coropi2xi2}
Let $\xi\ge x_0=10^{10}+19$ be a real number and $\xi_2$ be defined by 
\eqref{xik}. Then 
\begin{align}
  \notag
\frac{\xi^{3/2}}{3\sqrt 2\log \xi} \left(1+\frac{0.122}{\log \xi}\right)
&\le \pi_2^-(\xi_2) =\sum_{p < \xi_2} p^2 \le \pi_2(\xi_2)\\
 &\le   \frac{\xi^{3/2}}{3\sqrt 2\log \xi}
 \left(1+\frac{0.458}{\log \xi}\right).
 \label{pi2xi2}
\end{align}
\end{coro}

\begin{proof}
First, from \eqref{xik}, we observe that
\begin{equation}\label{xi2cube}
\frac{\xi_2^3}{\log \xi_2}=\xi_2\left(\frac{\xi_2^2-\xi_2}{\log \xi_2}\right)+
\frac{\xi_2^2}{\log \xi_2}=\xi_2\frac{\xi}{\log \xi}+\frac{\xi_2^2}{\log \xi_2}
\ge \xi_2 \frac{\xi}{\log \xi},
\end{equation}
whence, from \eqref{pi2minred} and \eqref{1slogxi2}, since $\xi \ge x_0$
and $\xi_2\ge x_0^{(2)}$ are assumed,
\[\pi_2^-(\xi_2) \ge \frac{\xi_2^3}{3\log \xi_2}
\left(1+\frac{0.248}{\log \xi_2}\right) \ge \xi_2\frac{\xi}{3\log
    \xi}\left(1+\frac{0.496}{\log \xi}\right) \]
and, from \eqref{x2366},
\begin{align*}
\pi_2^-(\xi_2)  &\ge \frac{\xi^{3/2}}{3\sqrt 2\log \xi}
     \left(1+\frac{0.496}{\log \xi}\right) \left(1-\frac{0.366}{\log
         \xi}\right)\\
    &= \frac{\xi^{3/2}}{3\sqrt 2\log \xi}
     \left(1+\frac{0.13}{\log \xi}-\frac{0.496\times 0.366}{\log^2 \xi}\right) \\
&\ge  \frac{\xi^{3/2}}{3\sqrt 2\log \xi}
     \left(1+\frac{0.13}{\log \xi}-\frac{0.496\times 0.366}{(\log x_0)\log \xi}\right) 
\ge \frac{\xi^{3/2}}{3\sqrt 2\log \xi}
\left(1+\frac{0.122}{\log \xi}\right),
\end{align*}
which proves the lower bound of \eqref{pi2xi2}.
\medskip

To prove the upper bound, as \eqref{xk1sk} implies $\xi_2 \le
\sqrt{\xi}$ and $\xi_2/\log \xi_2\le 2\sqrt \xi/\log \xi$, 
from \eqref{xi2cube}, we observe that
\begin{multline*}
\frac{\xi_2^3}{\log \xi_2}
=\xi_2\frac{\xi}{\log \xi}+\frac{\xi_2^2}{\log \xi_2}
=\xi_2\frac{\xi}{\log \xi}+\frac{\xi_2^2-\xi_2}{\log \xi_2}
   +\frac{\xi_2}{\log \xi_2}\\
=\frac{\xi}{\log \xi}(\xi_2+1) +\frac{\xi_2}{\log \xi_2}
\le \frac{\xi}{\log \xi}\Big(\xi_2+1+ \frac{2}{\sqrt \xi}\Big)
\end{multline*}
and, from \eqref{x2345} and \eqref{1slogxi2},
 \begin{align*}
 \frac{\xi_2^3}{\log \xi_2} 
&\le \frac{\xi}{\log \xi}\left(1+ \frac{2}{\sqrt \xi}+\sqrt{\frac{\xi}{2}}
  \left(1-\frac{\log 2}{2\log \xi}\right)\right)\\
&=\frac{\xi^{3/2}}{\sqrt 2 \log \xi}\left(1-\frac{1}{\log \xi}
  \left(\frac{\log 2}{2}-\frac{\sqrt 2(1+2/\sqrt \xi)\log^2 \xi}{\sqrt \xi}\right) \right)\\
& \le \frac{\xi^{3/2}}{\sqrt 2 \log \xi}\left(1-\frac{1}{\log \xi}
  \left(\frac{\log 2}{2}
   -\frac{\sqrt 2(1+2/\sqrt x_0)\log^2 x_0}{\sqrt x_0}\right) \right)\\
& \le \frac{\xi^{3/2}}{\sqrt 2 \log \xi}\left(1-\frac{0.339}{\log \xi}\right).
 \end{align*}
Further, from \eqref{1slogxi2}, we have $0.385/\log \xi_2\le
2.07\times 0.385/\log \xi\le 0.797/\log \xi$, whence
from \eqref{pi2majred},
\begin{multline*}\pi_2(\xi_2)\le \frac{\xi_2^3}{3 \log \xi_2}
  \left(1+\frac{0.385}{\log \xi_2}\right)
\le \frac{\xi^{3/2}}{3\sqrt 2 \log \xi}\left(1-\frac{0.339}{\log \xi}\right)
  \left(1+\frac{0.797}{\log \xi}\right)\\
\le \frac{\xi^{3/2}}{3\sqrt 2 \log \xi}\left(1+\frac{0.458}{\log
    \xi}\right),
\end{multline*}
which completes the proof of Corollary \ref{coropi2xi2}.
\end{proof}

\subsection{The additive excess and the multiplicative excess}\label{parexcess}

\subsubsection{The additive excess}\label{paraddexc}

Let $N$ be a positive integer and $Q(N)=\prod_{p\mid N} p$ be the 
squarefree part of $N$. The additive
excess $E(N)$ of $N$ is defined by
\begin{equation}
  \label{E}
  E(N)=\ell(N)-\ell(Q(N))=\ell(N)-\sum_{p\mid N} p=\sum_{p\mid N} (p^{v_p(N)}-p).
\end{equation}
If $N'$ and $N''$ are two consecutive superchampion numbers of common
parameter $\rho=\xi/\log \xi$ (cf. Proposition \ref{propNrho}), 
then, from \eqref{E} and \eqref{Nrho},
\begin{equation}\label{ellN'E}
\ell(N')=\sum_{p\mid N'} p +E(N')= \sum_{p< \xi} p +E(N').
\end{equation}

\begin{prop}\label{propEN}
Let $n$ be an integer satisfying $n\ge \nu_0$ (defined by \eqref{n0})
and $N'$ and $\xi$ defined by Defintion \ref{nsuperch} so that $\xi \ge
x_0=10^{10}+19$ holds, then the additive excess $E(N')$ satisfies
\begin{equation}
  \label{EN}
\frac{\xi^{3/2}}{3\sqrt 2 \log \xi}\left(1
+\frac{0.12}{\log \xi}\right)\le E(N') \le
  \frac{\xi^{3/2}}{3\sqrt 2 \log \xi}\left(1+\frac{0.98}{\log \xi}\right).
\end{equation}
\end{prop}

\begin{proof}
With $J$ defined by \eqref{J}, from
\eqref{Nrho}, we have
\begin{equation}
  \label{EN1}
E(N') = \sum_{p\mid N'} \left(p^{v_p(N')}-p\right)=
\sum_{j=2}^J\;\;\sum_{\xi_{j+1}\le p < \xi_j} (p^j-p).
\end{equation}
For an asymptotic estimates of $E(N')$ see below \eqref{ENxi2}.

\paragraph{The lower bound.} From \eqref{EN1}, we deduce
\begin{equation}\label{ENmin}
E(N') \ge \sum_{p<\xi_2} p^2 -\sum_{p\le \xi_2} p=\pi_2^-(\xi_2)-\pi_1(\xi_2).
\end{equation}
As $\xi_2\ge x_2^{(0)}> 69588$ holds, from \eqref{xsk1sk}
we have $\xi_2^2\le \xi/2$ and
from \eqref{pi126}, $\pi_1(\xi_2) \le \xi_2\pi(\xi_2)\le
1.26 \xi_2^2/\log \xi_2 \le 0.63\xi/\log 69588\le 0.057\xi$.
From
\eqref{ENmin} and \eqref{pi2xi2}, it follows that
\begin{align*}
  E(N') &\ge \frac{\xi^{3/2}}{3\sqrt 2 \log \xi}
  \left(1+\frac{0.122}{\log \xi}\right) -0.057 \xi\\
&=\frac{\xi^{3/2}}{3\sqrt 2 \log \xi}\left(1+\frac{1}{\log \xi}
  \left(0.122-\frac{0.057\times 3\sqrt 2 \log^2 \xi}{\sqrt \xi}\right) \right)\\
&\ge \frac{\xi^{3/2}}{3\sqrt 2 \log \xi}\left(1+\frac{1}{\log \xi}
  \left(0.122-\frac{0.057\times 3\sqrt 2 \log^2 x_0}{\sqrt x_0}\right) \right)\\
&\ge \frac{\xi^{3/2}}{3\sqrt 2 \log \xi}\left(1+\frac{0.12}{\log \xi}\right)
\end{align*}
which proves the lower bound of \eqref{EN}.
\medskip

\bfni{The upper bound.} Let us consider an integer $j_0$, $3\le j_0 \le
29$, that will be fixed later; \eqref{EN1} implies
\begin{equation}
  \label{EN2}
E(N')=\sum_{p<\xi_{j_0}}\left(p^{v_p(N')}-p\right)+\sum_{j=2}^{j_0-1}\;\;
\sum_{\xi_{j+1} \le p < \xi_j} (p^j-p) \le S_1+S_2.
\end{equation}
with
\[S_1=\sum_{p<\xi_{j_0}} p^{v_p(N')} \qtx{and} 
S_2=\sum_{j=2}^{j_0-1}\pi_j(\xi_j).\]
Let $\rho=\xi/\log \xi$ be the common parameter of $N'$ and
$N''$. From \eqref{Nrho}, for $p<\xi_{j_0}$, we have $p^{v_p(N')}\le \rho(\log
p)/(1-1/p)$ so that Lemma \ref{lemW} leads to 
$S_1\le \rho W(\xi_{j_0})\le \om\rho\xi_{j_0}$ with $\om=1.000014$ if $j_0\le 10$
and $\om=1.346$ if $j_0\ge 11$, since $x^{(0)}_{11}=6.55 < 7.32< x^{(0)}_{10}=7.96$.
In view of applying \eqref{xsk1sk} and \eqref{xkksk}, we set $\beta_j=j$ 
for $2\le j\le 8$ and $\beta_j=j(1-1/x_j^{(0)})$ (with $x_j^{(0)}$
defined by \eqref{xj0}) for $9\le j \le 29$. Therefore, from Lemma
\ref{lemW}, \eqref{xsk1sk} and \eqref{xkksk}, we get
\begin{multline}\label{EN3}
S_1\le \om\frac{\xi\xi_{j_0}}{\log \xi} 
\le \frac{\om \xi^{1+1/j_0}}{\beta_{j_0}^{1/j_0}\log \xi}
= \frac{\xi^{3/2}}{3\sqrt 2 \log^2 \xi}
 \left(\frac{3\sqrt 2\, \om \log \xi}{\beta_{j_0}^{1/j_0}\xi^{1/2-1/j_0}} \right)\\
\le \frac{\xi^{3/2}}{3\sqrt 2 \log^2 \xi}
  \left(\frac{3\sqrt 2\, \om\log x_0}{\beta_{j_0}^{1/j_0}x_0^{1/2-1/j_0}}\right). 
\end{multline}
In view of applying \eqref{majpi3}--\eqref{majpi6}, we set
$\al_3=0.271$, $\al_4=0.237$, $\al_5=0.226$ and, for $j\ge 6$,
$\al_j=(1+(2/3)^j)(\log 3)/3$. Therefore, for $3\le j \le 29$, it follows
from \eqref{majpi3}--\eqref{majpi6}, \eqref{xsk1sk} and \eqref{xkksk}
that $\pi_j(\xi_j)\le \al_j \xi_j^{j+1}/\log \xi_j$,  $\xi_j \le
(\xi/\beta_j)^{1/j}$ and
\begin{equation}
  \label{EN4}
 \pi_j(\xi_j)\le \al_j\frac{\xi_j^{j+1}}{\log \xi_j} \le 
\frac{\al_j (\xi/\beta_j)^{1+1/j}}{(1/j)\log (\xi/\beta_j)}=
\frac{j\al_j (\xi/\beta_j)^{1+1/j}}{(\log \xi)(1-(\log \beta_j)/\log
  \xi)}\le \ga_j\frac{\xi^{1+1/j}}{\log \xi}
\end{equation}
with 
\[\ga_j=\frac{j\al_j}{\beta_j^{1+1/j}(1-(\log \beta_j)/\log x_0)}.\]
Further, for $3\le j\le j_0-1$, we have
\[\ga_j\frac{\xi^{1+1/j}}{\log \xi}\frac{3\sqrt 2\log^2
  \xi}{\xi^{3/2}}=
\frac{3\sqrt 2\ga_j\log \xi}{\xi^{1/2-1/j}} \le \de_j\qtx{with} 
\de_j=\frac{3\sqrt 2\ga_j\log x_0}{x_0^{1/2-1/j}}\]
which implies from \eqref{EN4}
\begin{equation}
  \label{pixij}
  \pi_j(\xi_j)\le \frac{\de_j\xi^{3/2}}{3\sqrt 2 \log^2 \xi}.
\end{equation}
From the definition of $S_2$, from \eqref{pi2xi2} and from
\eqref{pixij}, one gets
\begin{multline*}
  S_2\le \sum_{j=2}^{j_0-1}\;\;\sum_{\xi_{j+1} \le p < \xi_j}
  p^j\le \sum_{j=2}^{j_0-1} \pi_j(\xi_j)\\
  \le
  \frac{\xi^{3/2}}{3\sqrt 2 \log
  \xi}\left(1+\frac{1}{\log \xi}\left(0.458+\sum_{j=3}^{j_0-1}
    \de_j\right)\right).
\end{multline*}
Finally, from \eqref{EN2} and \eqref{EN3}, we conclude
\[E(N')\le \frac{\xi^{3/2}}{3\sqrt 2 \log
  \xi}\left(1+\frac{1}{\log \xi}
\left(\frac{3\sqrt 2\,\om \log x_0}{\beta_{j_0}^{1/j_0}x_0^{1/2-1/j_0}} 
+0.458+\sum_{j=3}^{j_0-1} \de_j\right)\right)\]
which, by choosing $j_0=6$, completes the proof of \eqref{EN} (cf. \cite{web}). 
\end{proof}

\begin{rem}
  When $n=\nu_0$, $N'=N_0'$ is given by \eqref{N'0} and
  $E(N_0')=10\,517\,469\,635\,602$. Observing that $E(N'_0)$ is equal to
$\dfrac{x_0^{3/2}}{3\sqrt 2\log x_0}\Big(1+\dfrac{0.632\ldots}{\log x_0}\Big)$ 
shows that the constant $0.98$ in \eqref{EN} cannot be shortened below $0.632$.
\end{rem}

\subsubsection{The multiplicative excess}\label{parmulexc}

Let $N$ be a positive integer and $Q(N)=\prod_{p\mid N} p$ be the 
squarefree part of $N$. The multiplicative
excess $E^*(N)$ of $N$ is defined by
\begin{equation}
  \label{E*}
  E^*(N)=\log\left(\frac{N}{Q(N)}\right)=\log N-\sum_{p\mid N} \log p=
\sum_{p\mid N} (v_p(N)-1)\log p.
\end{equation}
If $N'$ and $N''$ are two consecutive superchampion numbers of common
parameter $\rho=\xi/\log \xi$ (cf. Proposition \ref{propNrho}), 
then from \eqref{Nrho},
\begin{equation}\label{ellN'E*}
\log N'=\sum_{p\mid N'} \log p +E^*(N')=\sum_{p< \xi} \log p +E^*(N').
\end{equation}

\begin{prop}\label{propEN*}
Let $n$ be an integer satisfying $n\ge \nu_0$ (defined by \eqref{n0})
and $N'$ and $\xi$ defined by Definition \ref{nsuperch} so that $\xi \ge
x_0=10^{10}+19$ holds. Then the multiplicative excess $E^*(N')$ satisfies
\begin{equation}
  \label{EN*}
\sqrt{\frac{\xi}{2}}\left(1-\frac{0.521}{\log \xi}\right) 
\le E^*(N') \le
\sqrt{\frac{\xi}{2}}\left(1+\frac{0.305}{\log \xi}\right) \le 0.72 \sqrt{\xi}.   
\end{equation}
\end{prop}

\begin{rem}
  When $n=\nu_0$, $N'=N_0'$ is given by \eqref{N'0} and
  $E^*(N_0')=70954.46\ldots=\sqrt{x_0/2}(1+(0.079385\ldots)/\log x_0)$,
  so  that the constant $0.305$ in \eqref{EN*} cannot be shortened below $0.079$.
\end{rem}

\begin{proof}
{\bf The lower bound}. 
From \eqref{E*}, \eqref{Nrho} and \eqref{J}, we may write
\begin{equation}
  \label{E*N1}
  E^*(N')=\sum_{j\ge 2}\;\;\sum_{\xi_{j+1}\le p < \xi_j}(j-1)\log p=
\sum_{j= 2}^J\theta^-(\xi_j)
\end{equation}
with $\theta^-(x)=\sum_{p<x} \log p$.
From \eqref{E*N1}, we deduce $E^*(N')\ge \theta^-(\xi_2)$ which, 
from \eqref{thxi2}, proves the lower bound of \eqref{EN*}.
\medskip

\bfni{The upper bound.} From \eqref{E*N1} and \eqref{J}, it follows
that
\begin{equation}
  \label{E*N3}
  E^*(N')\le \sum_{j=2}^J \theta(\xi_j).
\end{equation}
Let us fix $j_0=26$. From \eqref{thx<x1} with $\eps=7.5\times
10^{-7}$, we write
\begin{equation}
  \label{E*N4}
  \sum_{j=3}^{J} \theta(\xi_j)\le 
(1+\eps)\left(\sum_{j=3}^{j_0-1} \xi_j+\sum_{j=j_0}^{J}\xi_j\right)
= (1+\eps)(S_1+S_2).
\end{equation}

From \eqref{xsk1sk} and \eqref{xkksk} with $\beta_j=j$ for $3\le j \le 8$ and
$\beta_j=j(1-1/x_j^{(0)})$ for $9\le j \le j_0$, we get
\begin{multline}
  \label{S1maj}
  S_1=\sum_{j=3}^{j_0-1} \xi_j\le 
\sum_{j=3}^{j_0-1} \left(\frac{\xi}{\beta_j}\right)^{1/j}= 
\sqrt{\frac{\xi}{2}}\left(\frac{1}{\log \xi}\right)
\sum_{j=3}^{j_0-1}   \frac{\sqrt 2 \log \xi}{\beta_j^{1/j}\xi^{1/2-1/j}}\\
\le \sqrt{\frac{\xi}{2}}\left(\frac{1}{\log \xi}\right)
\sum_{j=3}^{j_0-1}   \frac{\sqrt 2 \log x_0}{\beta_j^{1/j}
  x_0^{1/2-1/j}} = 0.627703\ldots \sqrt{\frac{\xi}{2}}\left(\frac{1}{\log \xi}\right).
\end{multline}
Further, from \eqref{J} and \eqref{xkksk}, since $\xi_j$ is decreasing
on $j$, we have

\begin{multline}
  \label{S2maj}
  S_2=\sum_{j=j_0}^{J} \xi_j\le J\xi_{j_0} \le \frac{(\log
         \xi)\xi^{1/j_0}}{(\log 2)\beta_{j_0}^{1/j_0}}=
\sqrt{\frac{\xi}{2}}\left(\frac{1}{\log \xi}\right)
\left(\frac{\sqrt 2\log^2 \xi}{(\log 2)\beta_{j_0}^{1/j_0}\xi^{1/2-1/j_0}}\right)\\
\le \sqrt{\frac{\xi}{2}}\left(\frac{1}{\log \xi}\right)
\left(\frac{\sqrt 2\log^2 x_0}{(\log 2)\beta_{j_0}^{1/j_0}x_0^{1/2-1/j_0}}\right)
=0.022597\ldots \sqrt{\frac{\xi}{2}}\left(\frac{1}{\log \xi}\right).
\end{multline}
Finally, from \eqref{E*N3}, \eqref{thxi2}, \eqref{E*N4}, 
\eqref{S1maj} and \eqref{S2maj}, we conclude
\begin{align*}
 E^*(N') &\le  \sqrt{\frac{\xi}{2}}\left(1+\frac{1}{\log \xi}
   (-0.346+(1+\eps)(0.6278+0.0226))\right)\\
 &< \sqrt{\frac{\xi}{2}}\left(1+\frac{0.305}{\log \xi} \right)
\end{align*}
which ends the proof of Proposition \ref{propEN*}.
\end{proof}

\subsubsection{The number $s(n)$  of primes dividing $h(n)$ but not $N'$}\label{parsn}

Let $n\ge 7$  and $N',N''$ and $\xi$ defined by Definition
\ref{nsuperch}. Let us denote by $p_{i_0}$ the largest prime factor of
$N'$. Note, from \eqref{Nrho}, that $p_{i_0}$ is the largest prime $< \xi$. 
If $\xi$ is prime, we have $\xi=p_{i_0+1}$ while, if $\xi$ is not
prime $p_{i_0+1} > \xi$. In both cases, we have
\begin{equation}
  \label{pi0p1>xi}
 p_{i_0} < \xi \le p_{i_0+1} 
\end{equation}
From the definition of the additive excess \eqref{E}, we define 
$s=s(n)\ge 0$ by
\begin{equation}
  \label{defs}
p_{i_0+1}+\ldots + p_{i_0+s} \le n-\ell(N')+E(N') < p_{i_0+1} +\ldots + p_{i_0+s+1}. 
\end{equation}

\begin{prop}\label{propsn}
If $n\ge \nu_0$ (defined by \eqref{n0}) and $\xi$ defined in Definition
\ref{nsuperch}, we have
\begin{equation}
  \label{sn}
\frac{\sqrt \xi}{3\sqrt 2\log \xi}\left(1+\frac{0.095}{\log \xi}\right)
\le s \le
\frac{\sqrt \xi}{3\sqrt 2\log \xi}\left(1+\frac{1.01}{\log \xi}\right).
\end{equation}
\end{prop}

\begin{proof}
\bfni{The upper bound.}
Since $p_{i_0+1}\ge \xi$ and 
$n-\ell(N')< \ell(N'')-\ell(N')\le \xi$ hold (cf. \eqref{pi0p1>xi} and
\eqref{ellN'N''}), \eqref{defs} and \eqref{EN} imply
\begin{align}
  \nonumber
  s\xi &\le  n-\ell(N')+E(N')
\le \xi+\frac{\xi^{3/2}}{3\sqrt 2\log \xi}\Big(1+\frac{0.98}{\log
         \xi}\Big)\\
  \nonumber
&\le \frac{\xi^{3/2}}{3\sqrt 2\log \xi}\Big(1+\frac{1}{\log \xi}
              \big(0.98+\frac{3\sqrt 2\log^2 \xi}{\sqrt \xi}\big)\Big)\\
 \nonumber
&\le \frac{\xi^{3/2}}{3\sqrt 2\log \xi}\Big(1+\frac{1}{\log \xi}
  \big(0.98+\frac{3\sqrt 2\log^2 x_0}{\sqrt x_0}\big)\Big)\\
    \label{majs}
&\le \frac{\xi^{3/2}}{3\sqrt 2\log \xi}\left(1+\frac{1.01}{\log \xi}\right)
\end{align}
which yields the upper bound of \eqref{sn}.

\bfni{The lower bound.} First, from \eqref{majs},  we observe that
\begin{align}
  \label{sp1<}
  s+1 &\le \frac{\xi^{1/2}}{3\sqrt 2\log \xi}\left(1+\frac{1.01}{\log
        \xi}\right) +1\notag \\
&= \frac{\xi^{1/2}}{3\sqrt 2\log \xi}\left(1+\frac{1}{\log \xi}
  \left(1.01+\frac{3\sqrt 2\log^2 \xi}{\sqrt \xi}\right)\right) \notag
  \\
&\le \frac{\xi^{1/2}}{3\sqrt 2\log \xi}\left(1+\frac{1}{\log \xi}
       \left(1.01+\frac{3\sqrt 2\log^2 x_0}{\sqrt x_0}\right)\right)
       \notag \\
&\le \frac{\xi^{1/2}}{3\sqrt 2\log \xi}\left(1+\frac{1.033}{\log \xi}\right),
\end{align}
which implies
\begin{equation}
  \label{sp1<rac}
s+1 \le \frac{\xi^{1/2}}{3\sqrt 2\log x_0}\left(1+\frac{1.033}{\log x_0}\right)
< 0.011\sqrt \xi.   
\end{equation}
From Corollary \ref{coropixy}, the number of primes between $\xi$ and 
$\xi(1+0.045/\log^2 \xi)$ is $\ge 0.011 \sqrt \xi > s+1$ so that, 
in \eqref{defs}, we have
\begin{equation}
  \label{pi0psp1<}
  p_{i_0+s+1}\le \xi\left(1+\frac{0.045}{\log^2 \xi}\right)
\le \xi\left(1+\frac{0.045}{(\log x_0)(\log \xi)}\right)
\le \xi\left(1+\frac{0.002}{\log \xi}\right).
\end{equation}
From \eqref{N'n}, we get $n-\ell(N')\ge 0$. Therefore, from
\eqref{EN}, \eqref{defs} and \eqref{pi0psp1<}, we have
\[\frac{\xi^{3/2}}{3\sqrt 2\log \xi}\left(1+\frac{0.12}{\log \xi}\right)
  \le E(N') \le n-\ell(N')+E(N')
\le (s+1)\xi\left(1+\frac{0.002}{\log \xi}\right)\]
which yields
\begin{multline*}
  s+1
\ge  \frac{\xi^{1/2} (1+0.12/\log \xi)}{3\sqrt 2(\log \xi ) (1+0.002/\log \xi)} 
\ge \frac{\xi^{1/2}}{3\sqrt 2\log \xi}
  \left(1+\frac{0.12}{\log \xi}\right)\left(1-\frac{0.002}{\log \xi}\right)\\
\ge \frac{\xi^{1/2}}{3\sqrt 2\log \xi}
  \left(1+\frac{0.118}{\log \xi}-\frac{0.002\times 0.12}{(\log x_0)\log \xi} \right) 
\ge \frac{\xi^{1/2}}{3\sqrt 2\log \xi}\left(1+\frac{0.1179}{\log \xi}\right)
\end{multline*}
and
\begin{align}
 s &\ge  \frac{\xi^{1/2}}{3\sqrt 2\log \xi}\Big(1+\frac{0.1179}{\log \xi}\Big)-1\notag\\
&=\frac{\xi^{1/2}}{3\sqrt 2\log \xi}\left(1+\frac{1}{\log \xi}
\bigg(0.1179-\frac{3\sqrt 2 \log^2 \xi}{\sqrt \xi}\bigg)\right)\notag\\
&\ge \frac{\xi^{1/2}}{3\sqrt 2\log \xi}\left(1+\frac{1}{\log \xi}
\bigg(0.1179-\frac{3\sqrt 2 \log^2 x_0}{\sqrt x_0}\bigg)\right)\notag\\
&\ge \frac{\xi^{1/2}}{3\sqrt 2\log \xi}\left(1+\frac{0.095}{\log
  \xi}\right)
  \label{s>}
\end{align}
and the proof of Proposition \ref{propsn} is completed.
\end{proof}

From \eqref{pi0psp1<}, we deduce for $\xi\ge x_0$
\begin{equation}
  \label{logpi0p1<}
  \frac{ \log p_{i_0+s+1}}{\log \xi} \le 1+\frac{0.002}{\log^2 \xi}
\le 1+\frac{0.002}{(\log x_0)(\log \xi)} \le 1+\frac{0.0001}{\log \xi}
\end{equation}
and, from \eqref{sp1<},
\begin{align}
  (s+1)\log p_{i_0+s+1}
&\le  \frac{\sqrt \xi}{3\sqrt 2} \left(1+\frac{1.033}{\log \xi}\right)
  \left(1+\frac{0.0001}{\log \xi}\right)\notag \\
&= \frac{\sqrt \xi}{3\sqrt 2}
  \left(1+\frac{1.0331}{\log \xi}+\frac{0.0001033}{\log^2 \xi}\right)\notag \\
&\le \frac{\sqrt \xi}{3\sqrt 2}\left(1+\frac{1.0331}{\log \xi}+
  \frac{0.0001033}{(\log x_0)(\log \xi)}\right)\notag\\
&\le \frac{\sqrt \xi}{\sqrt 2}\left(\frac 13+\frac{0.345}{\log
  \xi}\right).
    \label{sp1log}
\end{align}
We shall also deduce from \eqref{s>} the following inequality 
valid for $\xi\ge x_0$:
\begin{multline}
  \label{sm1>}
(s-1)\log \xi
\ge  \frac{\sqrt \xi}{3\sqrt 2}\left(1+\frac{0.095}{\log \xi}\right)-\log \xi
=\frac{\sqrt \xi}{3\sqrt 2}
  \left(1+\frac{0.095}{\log \xi}-\frac{3\sqrt 2\log^2 \xi}{(\log \xi)\sqrt \xi}\right)\\
\ge   \frac{\sqrt \xi}{3\sqrt 2}
  \left(1+\frac{0.095}{\log \xi}-\frac{3\sqrt 2\log^2 x_0}{(\log \xi) \sqrt x_0}\right)
\ge \frac{\sqrt \xi}{3\sqrt 2}\left(1+\frac{0.0724}{\log \xi}\right).
\end{multline}

  \section{Some computational points}\label{parComputationalPoints}

  \subsection{Enumeration of superchampion numbers}\label{generate}

  Let us recall that $m$ is said a \emph{squarefull integer}
if, for every prime factor $p$ of $m$, $p^2$ divides $m$.
Each  $n \ge 1$ may be writen in a unique way $n = a b$, with $a$ squarefull,
$b$ squarefree and $a,b$ coprime. In the case where $n$ is a
superchampion number $N$,  we will say that
$a$ is the prefix\footnote{In \cite{DNZ} the term prefix is used with
  a  different meaning.}  of $N$.

Let $N \le N'$ be two consecutive superchampions, and $A, A'$ their
prefixes. In most of the cases $A'=A$  and $N' = p'N$ where $p'$ is
the prime following $\Pplus(N)$.

 When this is not the case, $N'  = q N$ and $A'=qA$ where $q$ is a prime facteur
 of $A$, or the prime following $\Pplus(A)$. In this case, we say that $N'$ is
 a superchampion of type 2.

  For example, let us consider the figure \ref{arraygn}. In this table
  two superchampions are of type $2$, $180\, 180$
  which is equal to 3 times its predecessor,
  and $360\,360$ which is equal to 2 times its predecessor.

  The superchampions of type $2$ are not very numerous. There are
  $455059774$ superchampions $N$ satisfying $12 \le N \le N'_0$
  (cf. \eqref{n0})  whose $7265$ are of type 2. We have precomputed the table
  TabT2, which, for each of these   $7265$ numbers $N$, keeps the
  triplet $(\ell(N), q, \log N)$, where $q$ is the quotient
  of $N$ by its predecessor (which, generally, is not of type $2$).
  For example, entries associated to the superchampions  $180\, 180$
  and $360\,360$ are the triplets  $(49,3, 12.101\ldots)$ and $(53, 2,
  12.794\ldots)$.
  With this table it is very fast to enumerate the increasing sequence
  of $(\ell(N), \log N)$ for all the superchampion  numbers.
  Let us associate to each superchampion $N$ the quadruple
  $(\ell(N), \log N, \Pplus(N), j)$, where $j$ is the
  smallest integer such that TabT2$[j][1] > \ell(N)$.

  The following function,   written in Python's programming language, computes the
  quadruple associated to the successor of $N$.
  
\begin{figure}[h]
  \begin{python}
    def next_super_ch(n, logN, pplusN, j):
       p =  next_prime(pplusN)
       if  n + p <= TabT2[j][1]:
           return (n+p, logN + log(p), p, j)
       else
          return(TabT2[j][1], logN + log(TabT2[j][2]), pplusN, j+1)
  \end{python}
  \caption{Enumeration of super-champion numbers}
  \label{enumsch}
\end{figure}

Using a prime generator function, which computes the sequence of
successive primes up to $n$ in time $\bigo{n\log \log n}$,
we wrote a C$^{++}$ function which computes
the pairs $(\ell(N), \log N)$, for all
superchampion numbers up to $N'_0$, in time about 22 seconds.

\subsection{Computing and bounding $g(n)$ and $h(n)$ on finite intervals}

\subsubsection{Computating an isolate value of $\log h(n)$ or $\log g(n)$}
The computation of an isolate value $\log h(n)$ by \eqref{hNG}
(cf. \cite[Section 8]{DNh1}) or $\log g(n)$ (by the algorithm
described in \cite{DNZ})  is
relatively slow. The table below shows the time of these computations
in ms.  for $n$ randomly choosen in intervals $[1,10^{j}]$ for
$j=9,12,15,16,17,18$ (on a MacBook 2016 computer).
\begin{center}
\begin{tabular}{|c|r|r|r|r|r|r|}
 \hline
  n & $10^9$ & $10^{12}$  & $10^{15}$ & $10^{16}$  & $10^{17}$ & $10^{18}$ \\
  \hline
 $\log$  h(n) &   5.55 &  5.78 &  6.74  & 7.38 &  8.39 & 9.80\\
  \hline
  $\log$ g(n) &  10.1 &  38.2 &  221. & 589. &  2\,467. &7\,980.\\
  \hline
\end{tabular}
\end{center}

For $n > 10^{16}$ the computation of an isolate value $\log g(n)$ takes a
few seconds, and it is impossible to compute more than some thousands
of these values.

\subsubsection{Bounding by slices}
We will need effectives bounds of $\log g(n)$ or $\log h(n)$ on intervals up to
$\nu_0=2.22 \cdot 10^{18}$ (cf. \eqref{n0}).  It is
impossible to compute a lot of these  values
for ordinary large integers $n$.  Nethertheless, by using
\texttt{next\_super\_ch},
we can enumerate quickly  the seqence $(N, \log N)$ of
superchampions and of their logarithms.
If, in the same time, we
enumerate the values $k(\ell(N))$ (cf. \eqref{kn}),
by using lemma \ref{boundslice} we get good estimates
of $\log h(n)$ and $\log g(n)$ on the intervals
$[\ell(N),\,\ell(N')]$, for values of $\ell(N)$ up to $\nu_0$.

\begin{lem}\label{boundslice}
  Let $N_1, N_2$ be two consecutives superchampion numbers.
  Let us define  $n_1=\ell(N_1)$,   $k_1=k(n_1)$, $m_1=n_1-\sigma_{k_1}$,
 $n_2=\ell(N_2)$, $k_2=k(n_2)$, $m_2=n_2-\sigma_{k_2}$
 and $q$ as the smalllest prime not smaller than  $p_{k+1} - m_1$.
 Then
 \begin{equation}\label{sliceminh}
  \log h(n_1) \ge  \theta(p_{k_1+1})  - \log q
\end{equation}

\begin{equation}\label{slicemajh}
  \log h(n_2) \le \theta(p_{k_2+1}) - \log (p_{k_2}-m_2)
\end{equation}
\end{lem}

\begin{proof}
The lower bound for $h(n_1) = N_{k_1} G(p_{k_1},m_1)$  (cf. \eqref {hNG})
 comes from  \cite[Proposition 8]{DNZ}), applied to $G(p_{k_1},m_1)$,
 and the upper bound for $h(n_2)$ from the same proposition
applied to $G(p_{k_2},m_2)$.
 \end{proof}

   \subsubsection{A dichotomic algorithm}\label{pardicho}
  We recall the algorithm \texttt{ok\_rec(n1,   n2)},
  presented  in (\cite[Section 4.9]{DNh3}),  which we
  will use several times. 
  Let us suppose that \texttt{ok(n)} is a  boolean function
  with the following side effect: when it returns false, before
  returning, it prints \emph{n does not satisfy property ok}.
 We suppose that we also have at our disposal
  a boolean function \texttt{good\_interval(n1, n2)}
  such that, when it returns true, the property $ok(n)$
  is satisfied by all $n \in  \undeux$; in other words
  $good\_intervall(n_1, n_2)$ is a sufficient condition
  (most often not necessary)
  ensuring that $ok(n)$ is true on $\undeux$.

Then the procedure \texttt{ok\_rec(n1,   n2)} returns true if and only $ok(n)$ is true
for every $n \in \undeux$, and,  when it return false, before
returning,  it prints the value  of  the largest $n$
in $[n_1, n_2]$ which does not satisfy \texttt{ok(n)}.

This procedure is used in Sections \ref{parproof172} (Theorem \ref{thmhgMm}),
\ref{parproofthmgshii} (Theorem \ref{thmgsh} (ii) and (iii)),
\ref{parproofthmgshv} (Theorem \ref{thmgsh} (v))  and \ref {thhmgnHR3}
(Theorem  \ref{thmgnHR} (iii)).

\section{Proof of Theorem \ref{thmhgMm}}\label{parProofthmhgMm}

 \subsection{Estimates of $\xi$ in terms of $n$}\label{xin} 

\begin{lem}\label{lemiter}
Let $n$ and $\de$ be two numbers satisfying $n > e^6=403.42\ldots$
and $1\le \de\le 2$. Let us define $f:[\sqrt n,n] \lr \R$ by
\[f(t)=f_{n,\de}(t)=\sqrt{n(2\log t-\de)}.\]

(i) $f([\sqrt n,n])$ is included in $(\sqrt n,n)$.

(ii) $f$ is increasing and $f'(t) \le 1/2$ holds.

(iii) The equation $t=f(t)$ has a unique root $R=R(n,\de)$ in $(\sqrt
n,n)$. If $R<t\le n$, then we have $R<f(t)< t$ while, if $\sqrt n
\le t < R$, $R>f(t) >t$ holds.
\end{lem}

\begin{proof}
we have $f(\sqrt n)=\sqrt{n(\log n-\de)}\ge \sqrt{n(\log n-2)}\ge
\sqrt{4n}> \sqrt n$. On the other hand, we have $f(n)=\sqrt{n(2\log
  n-\de)} < \sqrt{2n\log n}$ and by using the inequality $\log n \le
  n/e$, $f(n)\le (\sqrt{2/e})n$, which completes the proof of (i).

The derivative $f'(t)=\frac{\sqrt n}{t\sqrt{2\log t-\de}}$ is clearly
  positive and we have
\[f'(t) \le \frac{\sqrt n}{\sqrt n\sqrt{\log n-2}} \le
  \frac{1}{\sqrt{\log(e^6)-2}}=\frac12,\]
which proves (ii).

From (ii), the derivative of $t\mapsto f(t)-t$ is negative while, from
(i), $f(t)-t$ is positive for $t=\sqrt n$ and negative for $t=n$,
whence the existence of the root $R$ and for $\sqrt n \le t \le n$,
the equivalences
\[ f(t)< t \;\; \Llr \;\; t >f(t) > R \qtx{and} f(t) > t\;\; \Llr\;\; t <
f(t) <R,\]
which proves (iii).
\end{proof}

Let us recall that $k=k(n)$ is defined by \eqref{kn} and let us set 
\begin{equation}\label{x(n)}
x=x(n)=p_{k+1}
\end{equation}
so that 
\begin{equation}\label{npi1x}
\pi_1(x)-x \le n < \pi_1(x)
\end{equation}
holds. Further, $n\ge 7$ being given, one defines $N'$ and $N''$ by
Definition \ref{nsuperch}, 
$\rho=\rho(n)$ is the common parameter of $N'$ and $N''$
(cf. Proposition \ref{propNrho}) and $\xi=\xi(n)$ is defined by
$\rho=\xi/\log \xi$. If $p_{i_0}$ denotes the largest prime factor of $N'$,
from \eqref{pi0p1>xi}, we have $p_{i_0} < \xi \le p_{i_0+1}$ and, from \eqref{Nrho},
$\ell(N')\ge \pi_1(p_{i_0})$. Therefore, from \eqref{N'n} and
\eqref{npi1x}, we get
\[\pi_1(p_{i_0})\le \ell(N') \le n <  \pi_1(x) =\pi_1(p_{k+1}), \]
which implies $p_{i_0} < p_{k+1}$, $p_{i_0+1} \le  p_{k+1}$ and from \eqref{pi0p1>xi},
\begin{equation}\label{xix}
\xi=\xi(n) \le p_{i_0+1} \le p_{k+1}=x=x(n).
\end{equation}
Note that  in $N''$ (cf. \eqref{Nrhop}), from Corollary \ref{coropjxi}, all
prime powers dividing $N''$ do not exceed $\xi$, so that
$n < \ell(N'')\le \xi \pi(\xi) \le \xi^2$ and, with \eqref{xix},
\begin{equation}\label{lognxi}
\log n \le 2\log \xi\le 2\log x.
\end{equation}

Proposition \ref{propxin} improves on Lemma 2.8 of \cite{DNh3}.
\begin{prop}\label{propxin}
For $n \ge \nu_0$ (defined by \eqref{n0}), 
\begin{equation}
  \label{xinm}
  \sqrt{n\log n}\left(1+\frac{\log \log n-1}{2\log n} -
\frac{(\log \log n)^2}{8\log^2 n}+0.38\,\frac{\log \log n}{\log^2
  n}\right)\le \xi\le x
\end{equation}
while, for $n\ge \pi_1(x_0)= 2\,220\,822\,442\,581\,729\,257 = 2.22\ldots 10^{18}$,
\begin{equation}
  \label{xinM}
\xi \le x\le \sqrt{n\log n}\left(1+\frac{\log \log n-1}{2\log n} -
\frac{13\, (\log \log n)^2}{10000\,\log^2 n}\right).
\end{equation}
\end{prop}

\begin{proof}
{\bf The lower bound \eqref{xinm}.} 
First, from the definition of $\nu_0$ (cf. \eqref{n0}), $\xi$, $N'$ and $N''$ 
(cf. Definition \ref{nsuperch}), it follows that $n\ge \nu_0$ implies
$\xi \ge x_0$ and that, from \eqref{ellN'N''} and \eqref{N''xiN'},
\[n< \ell(N'')=\ell(N') +\ell(N'')-\ell(N')=\ell(N')+\rho \log(N''/N')
  \le \ell(N')+\xi\]
which, from  \eqref{ellN'E} and \eqref{EN}, yields
\begin{eqnarray}
  \label{xnxi1}
  n &\le& \ell(N')+\xi=\sum_{p<\xi} p +E(N')+\xi \le \pi_1(\xi)+\xi+
\frac{\xi^{3/2}}{3\sqrt 2 \log \xi}\left(1+\frac{0.98}{\log
      \xi}\right)\notag\\
&=& \pi_1(\xi)+\frac{\xi^2}{\log^4 \xi}\left(\frac{\log^4
    \xi}{\xi}+\frac{\log^3 \xi}{3\sqrt 2 \sqrt
    \xi}\left(1+\frac{0.98}{\log \xi}\right)\right) \notag\\
&\le& \pi_1(\xi)+\frac{\xi^2}{\log^4 \xi}\left(\frac{\log^4
    x_0}{x_0}+\frac{\log^3 x_0}{3\sqrt 2 \sqrt
    x_0}\left(1+\frac{0.98}{\log x_0}\right)\right) \notag\\
&\le&  \pi_1(\xi)+0.0301\frac{\xi^2}{\log^4 \xi}.
\end{eqnarray}
Further, as $107/160+0.0301 < 7/10$ holds, it follows from
\eqref{pi1maj} and \eqref{xnxi1}, that
\begin{equation}\label{n7/10}
 n\le 
\frac{\xi^2}{2\log \xi}+\frac{\xi 2}{4\log^2 \xi}+
\frac{\xi^2}{4\log^3 \xi}+\frac{7\;\xi^2}{10 \log^4 \xi}.
\end{equation}
Let us consider the polynomial
\begin{align*}
  P&=\left(\frac t2+\frac{t^2}{4} +\frac{t^3}{4}+\frac{7\,t^4}{10}\right)
\left(\frac 2t    -1-0.584\,t\right)\\
&=1-0.042\,t^2+1.004\,t^3-0.846\,t^4-0.4088\,t^5.
\end{align*}
The polynomial $P-1$ has a double root in $0$ and three other roots $-2.92\ldots,\\
0.0434574\ldots$ and $0.809\ldots$. Therefore, $P\le 1$ holds for $0\le
t \le 1/\log x_0=0.0434294\ldots$ and
\eqref{n7/10} implies
\begin{equation}\label{n584}
n\le \frac{\xi^2}{2\log \xi-1-0.584/\log \xi} \qtx{for} \xi \ge x_0.
\end{equation} 
Therefore, from \eqref{lognxi}, \eqref{n584} yields
\[n \le \frac{\xi^2}{2\log \xi-1-1.168/\log n},\]
which implies
\begin{equation}
  \label{xifxi}
\xi \ge  f(\xi) \qtx{with} f(t)=f_{n,\de}(t)=\sqrt{n(2\log t-\de)}  
\end{equation}
with
\begin{equation}\label{fdelta}
1 < \de=1+\frac{1.168}{\log n}\le 1+\frac{1.168}{\log \nu_0}\le 1.03.
\end{equation}
From Lemma \ref{lemiter} (ii), the equation $t=f(t)$ 
has one root $R\in (\sqrt n,n)$.  Let us set 
\begin{equation}\label{t1x}
t_1= \sqrt{n\log n\left(1+\frac{\log \log n-\de}{\log n}\right)}.
\end{equation}
We have $t_1\in (\sqrt n,n)$ and
\begin{equation}\label{ft1}
f(t_1)=\sqrt{n(\log n+\log \log n +\log(1+u)-\de)} \qtx{with}
u=\frac{\log \log n-\de}{\log n}.
\end{equation}
As $n\ge \nu_0$ holds, $u$ is positive and 
$f(t_1)^2-t_1^2=n\log(1+u) > 0$ so that $f(t_1) > t_1$ and, 
from Lemma \ref{lemiter} (iii), the root $R$ satisfies $R > f(t_1)$.
But, from \eqref{xifxi},  $\xi \ge f(\xi)$, which implies
\begin{equation}\label{xiR}
\xi \ge  R \ge f(t_1). 
\end{equation}
By Taylor formula, since the third derivative 
of $t\mapsto \log(1+t)$ is positive, we have $\log(1+u)\geq u-u^2/2$.
For convenience, from now on, we write $L$ for $\log n$,  $\la$ for
$\log n$, $L_0$ for $\log \nu_0=42.244 414\ldots$, and $\la_0$ for $\log\log
\nu_0=3.743 472\ldots$ With \eqref{fab}, 
\[\frac{u^2}{2} =\frac{(\la-\de)^2}{2L^2}\leq
\frac{\la-1}{2L}\left(\frac{\la-\de}{L}\right)\leq
\frac{\la_0-1}{2L_0}\left(\frac{\la-\de}{L}\right)\leq 0.04\ \frac{\la-\de}{L}\]
and 
\[\log(1+u)\ge u-\frac{u^2}{2} \ge \frac{\la-\de}{L}-0.04\, 
\frac{\la-\de}{L}=0.96\, \frac{\la-\de}{L},\]
which, from \eqref{xiR} and \eqref{ft1}, yields
\begin{equation}
  \label{x>log1+v}
  \xi \ge f(t_1)= \sqrt{n(\log n)(1+v)}
\end{equation}
with
\begin{equation}
  \label{v}
\frac{\la-\de}{L}\left(1+\frac{0.96}{L}\right) 
\le v = \frac{\la-\de}{L}+\frac{\log(1+u)}{L}
 \le \frac{\la-\de}{L} + \frac{ u}{L}
= \frac{\la-\de}{L}\left(1+\frac{1}{L}\right).
\end{equation}
For $v>0$, by Taylor formula, since the third derivative 
of $t\mapsto \sqrt{1+t}$ is positive, we have $\sqrt{1+v} \ge
1+v/2-v^2/8$ and we need an upper bound for $v^2/8$. From
\eqref{fdelta}, we get
\begin{align}
\frac{v^2}{8} &= \frac{(\la-\de)^2}{8L^2}\left(1+\frac{1}{L}\right)^2
  = \frac{(\la-\de)^2}{8L^2}\left(1+\frac{2}{L}\left(1+\frac{1}{2L}\right)\right) \notag \\
&\le \frac{(\la-\de)^2}{8L^2}\left(1+\frac{2}{L}\left(1+\frac{1}{2L_0}\right)\right)\notag \\
&\leq \frac{(\la-\de)^2}{8L^2}\left(1+\frac{2.03}{L}\right)
 \le \frac{(\la-1)^2}{8L^2}+ \frac{2.03\,\la^2}{8\,L^3}\notag \\
&= \frac{\la^2-2\la+1}{8L^2}+\frac{\la}{8L^2}\left(\frac{2.03\;\la}{L}\right) 
   \le \frac{\la^2}{8L^2}-\frac{2\la}{8L^2}+\frac{\la}{8\la_0L^2}
   +\frac{\la}{8L^2}\left(\frac{2.03\;\la_0}{L_0}\right) \notag \\
&\le\frac{\la^2}{8L^2}+\frac{\la}{8L^2}(-2+0.27+0.18)
   \le \frac{\la^2}{8L^2}-0.19\frac{\la}{L^2}.
      \label{v2sur8}
\end{align}
Finally, from \eqref{x>log1+v}, \eqref{v}, \eqref{v2sur8} and \eqref{fdelta},
\begin{eqnarray}
  \label{xmino}
\frac{\xi}{\sqrt{n\log n}} &\ge& \sqrt{1+v}  \ge
1+\frac v2-\frac{v^2}{8}\notag\\
&\ge& 1+\frac{\la-\de}{2L}+0.48 \frac{\la-\de}{L^2}-
  \frac{\la^2}{8L^2}+0.19\frac{\la}{L^2}\notag\\
&\ge& 1+\frac{\la-1}{2L}-\frac{1.168}{2L^2}+0.48\, \frac{\la-1.03}{L^2}-
  \frac{\la^2}{8L^2}+0.19\frac{\la}{L^2}\notag\\
&\ge& 1+\frac{\la-1}{2L}-\frac{\la^2}{8L^2}
  +0.67 \,\frac{\la}{L^2}-\frac{(0.584+0.48\times 1.03)\la}{L^2\la_0}\notag\\
&\ge& 1+\frac{\la-1}{2L}-\frac{\la^2}{8L^2}+\frac{\la}{L^2}(0.67-0.29)\notag \\
&=& 1+\frac{\la-1}{2L}-\frac{\la^2}{8L^2}+0.38\,\frac{\la}{L^2},
\end{eqnarray}
which proves \eqref{xinm}.

\bfni{The upper bound \eqref{xinM}.}  We assume $x\ge x_0=10^{10}+19$. As
\[\frac{3x^2}{20\log^4 x}-x =\frac{x^2}{\log^4 x}
  \left(\frac{3}{20}-\frac{\log^4 x}{x}\right) \ge
\frac{x^2}{\log^4 x}\left(\frac{3}{20}-\frac{\log^4 x_0}{x_0}\right)
\ge 0.149\, \frac{x^2}{\log^4 x},\]
  \eqref{pi1min} and \eqref{npi1x} imply
\begin{equation}\label{n149}
n\ge \pi_1(x)-x \ge \frac{x^2}{2\log x}+\frac{x^2}{4\log^2 x}
+\frac{x^2}{4\log^3 x}+0.149\,\frac{x^2}{\log^4 x}.
\end{equation}
Let us set
\begin{multline*}
  Q=\left(\frac t2+\frac{t^2}{4}+\frac{t^3}{4}+0.149\,
      t^4\right)\left(\frac 2t -1-0.492\,t\right)\\
=1+0.004\,t^2-0.075\,t^3-0.272\,t^4-0.073308\,t^5.
\end{multline*}
The polynomial $Q-1$ has a double root in $0$ and three other roots
\[
  -3.4052\ldots,\
  -0.3508\ldots,\
  0.04567\ldots,
  \]
and  $Q\ge 1$ holds for $0\le
t \le 1/\log x_0=0.0434294\ldots$,which, from \eqref{n149}, proves
\begin{equation*}
n\ge \frac{x^2}{2\log x-1-0.492/\log x} \qtx{for} x \ge x_0.
\end{equation*} 
Further, \eqref{n149} implies $n\ge x^2/(2\log x)$, whence
\begin{multline*}
\log n \ge 2\log x-\log(2\log x)=( \log x)\left(2-\frac{\log (2
      \log x)}{\log x}\right)\\
\ge ( \log x)\left(2-\frac{\log (2\log x_0)}{\log x_0}\right)=
1.8336\ldots\log x
\end{multline*}
and, as $1.8336\times 0.492 \ge 0.902$, 
\begin{equation}\label{n902}
n\ge \frac{x^2}{2\log x-1-0.902/\log n} \qtx{for} n \ge \nu_0
\end{equation}
and 
\begin{equation}\label{xfx2}
x\le f(x) \text{ with } f(t)=f_{n,\de}(t)=\sqrt{n(2\log t-\de)},\
  b=0.902, \ \de=1+\frac{b}{\log n}.   
\end{equation}
This time, one chooses
\[t_2=A\sqrt{n\log n} \qtx{with} A=1+\frac{\la-1}{2L}
\le 1+\frac{\la_0-1}{2L_0} \le 1.033 \]
and one calculates
\[f(t_2)=\sqrt{B\ n\log n}\]
with
\[B=1+\frac{\la-\de}{L}+\frac 2L\log\left(1+\frac{\la-1}{2L}\right) \le
B'=1+\frac{\la-\de}{L}+\frac{\la-1}{L^2}.\]
We have
\[A^2-B\ge A^2-B'=\frac{1}{4L^2}(\la^2-6\la+5+4b)=
\frac{1}{4L^2}(\la^2-6\la+8.608)\]
and 
\[\la^2-6\la+8.608=0.011\la^2+(0.989\la^2-6\la+8.608).\]
The roots of the above trinomial are $2.327\ldots$ and $3.738\ldots<
\la_0$ so that it is positive for $\la \ge \la_0$ and one gets
$A^2-B\ge 0.011\la^2/(4L^2)$, $A> \sqrt B$ and
\begin{equation}\label{ARB}
A-\sqrt B=\frac{A^2-B}{A+\sqrt B}
\ge \frac{A^2-B}{2A}\ge \frac{0.011\la^2}{4\times 2.066\,L^2}
\ge 0.0013\frac{\la^2}{L^2},
\end{equation}
so that $t_2 > f(t_2)$ holds. By Lemma \ref{lemiter}, the root $R$
of the equation $t=f(t)$ satisfies $R <  f(t_2)$ and \eqref{xfx2}
implies
\begin{multline*}
  x \le  f(t_2) = \sqrt{B\ n\log n}=\sqrt{n\log n}(A-(A-\sqrt B))
  \\ \le \sqrt{n\log n}(A-0.0013\la^2/L^2)
\end{multline*}
which proves \eqref{xinM}.
\end{proof} 

\begin{coro}\label{corominxi}
For $n\ge \nu_0$,
\begin{equation}\label{minxicoro}
\sqrt{n\log n}\left(1+\frac{\log \log n-1.019}{2\log n} \right)\le \xi \le x
  \le \sqrt{n\log n}\left(1+\frac{\log \log n-1}{2\log n}\right).
\end{equation}
  \end{coro}

\begin{proof}
The upper bound follows from \eqref{xinM}. From \eqref{xinm},
\begin{multline*}
  \xi\ge \sqrt{n\log n}\left(1+\frac{\log \log n-y(\log \log n)}{2\log
      n} \right)
  \\\qtx{with} y(t)=1+(t^2/4-0.76\,t)\exp(-t).
\end{multline*}
The derivative $y'(t)=(-0.25\,t^2+1.26\,t-0.76)\exp(-t)$ vanishes for $t=0.7005\ldots$
and $t=4.339\ldots$ so that, for $t\ge \la_0$, $y(t)$ is maximal for
$t=4.339\ldots$ and its value is $1.01838\ldots$.
\end{proof}

\subsection{Proof of the lower bound \eqref{minhn} for $n \ge \pi_1(x_0)$.}\label{parproof171}
Let us recall that $x_0=10^{10}+19$, and let us suppose first that 
$n \ge \pi_1(x_0) = 2\,220\,822\,442\,581\,729\,257$, so that $x=x(n)=p_{k+1}$
defined by \eqref{x(n)} is $\ge x_0$. As  the function $h$ is
nondecreasing, from \eqref{dusart3} with $\al=1/2$, 
\begin{equation}
  \label{h>th}
\log h(n) \ge \log N_{k} =\theta(p_{k})=\theta(x)-\log x\ge x-
\frac{x}{2\log^3 x}-\log x.  
\end{equation}

Inequality \eqref{h>th} together with \eqref{fab} and \eqref{lognxi}
yield, by noting $L$ for $\log n$, $\la$ for $\log \log n$, $L'_0$ for
$\log \pi_1(x_0)$, $\la'_0$ for $\log \log \pi_1(x_0)$,
\begin{eqnarray*}
\frac{\log h(n)}{x} &\geq& 1-\frac{1}{2\log^3 x}-\frac{\log x}{x}
= 1-\left(\frac 12+\frac{\log^4 x}{x}\right)\frac{1}{\log^3 x}\notag\\
&\ge& 1-\left(\frac 12+\frac{\log^4 x_0}{x_0}\right)\frac{1}{\log^3 x}
\ge  1-\frac{0.500029}{\log^3 x}\notag\\
&\ge&   1-\frac{4.0003}{\log^3 n}
\ge  1  -\frac{4.0003\, \la}{(\log^2 n) \la'_0 L'_0}
\ge 1 -\frac{0.026\,\la}{L^2}
\end{eqnarray*}
and, from \eqref{xinm},
\begin{eqnarray}
  \label{hnmino}
 \frac{\log h(n)}{\sqrt{n\log n}}&\ge&   
\left(1+\frac{\la-1}{2L}-\frac{\la^2}{8L^2}+\frac{0.38\,\la}{L^2}\right)
\left(1-\frac{0.026\ \la}{L^2}\right) \notag\\
&\ge& 1+\frac{\la-1}{2L}-\frac{\la^2}{8L^2}+\frac{0.38\,\la}{L^2}
-\frac{0.026\ \la}{L^2}\left(1+\frac{\la}{2L}+\frac{0.38\,\la}{L^2}\right)\notag\\
&\ge& 1+\frac{\la-1}{2L}-\frac{\la^2}{8L^2}+\frac{0.38\,\la}{L^2}
-\frac{0.026\ \la}{L^2}\left(1+\frac{\la'_0}{2L'_0}+\frac{0.38\,\la'_0}{L_0^{'2}}\right)\notag\\
&\ge& 1+\frac{\la-1}{2L}-\frac{\la^2}{8L^2}+\frac{0.35\,\la}{L^2},
\end{eqnarray}
which proves \eqref{minhn} for $x\ge \pi_1(x_0)$.
\qed

\subsection{Proof of the lower bound \eqref{minhn}\label{parproof172} 
for $n < \pi_1(x_0)$.}
Let $\Phi_{u}$  defined by \eqref{PHIu} and $n_1 \le n_2$ such that  the following inequality
\begin{equation}\label{okrec1.7}
\log h(n_1) \ge \Phi_{1/8}(n_2) 
\end{equation}
is true.  Then,  by the non decreasingness of $h$ and $\Phi_{1/8}$,
$\log h(n) \ge \Phi_{1/8}(n)$ is true on the whole interval $\undeux$.
In particular, \eqref{minhn} is satisfyed on $[\sk, \sigma_{k+1}]$
if the
following inequality is true
\begin{equation}\label{goodk}
\theta(p_k) = \log h(\sigma_k) > \Phi_{1/8}(\sigma_{k+1}),  
\end{equation}
By enumerating $p_k$, $\sk$ and $\theta_k$  until
$p_{k+1} = x_0=10^{10}+19$ 
we remark that \eqref{goodk} is satisfyied for $k  \ge k_1 = 9\,018$.
This proves that inequality \eqref{minhn} is true for $n \ge \si_{k_1}
=398\,898\,277$.

It remains to compute the largest $n$ in $[2, \sigma_{k_1}]$ such that
$\Phi_{1/8}(n) \le \log h(n)$ fails.
This is done by dichotomy (cf. Section \ref{pardicho}), calling \texttt{ok\_rec(2, 398 898 277)}
with  $ok(n)$ which returns true if and and only if $\log h(n) \ge \Phi_{1/8}(n)$ and
$good\_interval(n_1, n_2)$ which returns true if and only if
\eqref{okrec1.7} is true.
This gives the largest $n$ in $[2, 398\,898\,277]$,  which does not satisfy
\eqref{minhn}, $ n=373\,623\,862$, and this call of \texttt{ok\_rec}
computes  $3577$ values of $good\_interval$ and $2$ values of $ok(n)$.
\qed

\subsection{Proof of the upper bound \eqref{maxgn} \label{parproof173}
for $n \ge \nu_0$.}
For $n\ge \nu_0$ (defined by \eqref{n0}), one defines $N'$, $N''$ and
$\xi$ by Definition \ref{nsuperch}.  
The inequalities $\xi \ge x_0$ and $N''\le \xi N'$ hold (cf. \eqref{N''xiN'}). 
From \eqref{N'n}, \eqref{ellN'E*} and \eqref{EN*},
\begin{align*}
\log g(n) &\le \log N''=\log N'+\log \frac{N''}{N'}
=\sum_{p < \xi}\log p+E^*(N')+\log \frac{N''}{N'}\\
&\le \theta(\xi)+0.72\sqrt \xi +\log \xi.
\end{align*}
Further, from \eqref{minxicoro}, with our notation $L=\log n, \la=\log
L$, $L_0=\log \nu_0,\la_0=\log L_0$,
\begin{multline}
  \label{logg1}
 \log g(n)\le \theta(\xi)+\sqrt \xi \left(0.72+\frac{\log \xi}{\sqrt \xi} \right)
\le \theta(\xi)+\sqrt \xi \left(0.72+\frac{\log x_0}{\sqrt{x_0}} \right)\\
\le \theta(\xi)+0.73\sqrt \xi 
\le \theta(\xi)+0.73(n\log n)^{1/4} \left(1+\frac{\la-1}{4L}\right)\\
\le \theta(\xi)+0.73(n\log n)^{1/4}\left(1+\frac{\la_0-1}{4L_0}\right)
\le \theta(\xi)+0.75(n\log n)^{1/4}.
\end{multline}
Now, we consider two cases, according to $\xi \le 10^{19}$ or not.
\begin{itemize}
\item
 If $x_0 \le \xi \le 10^{19}$, then \eqref{logg1}, \eqref{thx<x} and
  \eqref{xinM} imply
  \begin{align*}
\log g(n) &\le \xi+0.75(n\log n)^{1/4}
= \xi+\frac{\la^2}{L^2}\sqrt{n\log n}
\left(\frac{0.75L^{7/4}}{n^{1/4}\la^2}\right)\\
&\le \xi+\frac{\la^2}{L^2}\sqrt{n\log n}
\left(\frac{0.75L_0^{7/4}}{\nu_0^{1/4}\la_0^2}\right)\\
&\le \sqrt{n\log n}\left(1+\frac{\la-1}{2L}-\frac{\la^2}{L^2}
\left(\frac{13}{10^4}-10^{-3}\right)\right)\\
&= \sqrt{n\log n}
\left(1+\frac{\la-1}{2L}-\frac{3\la^2}{10^4L^2}\right)
\end{align*}
which proves \eqref{maxgn} for $x_0\le \xi \le 10^{19}$.

\item If $\xi > x_6=10^{19}$, then from \eqref{Nrho}, \eqref{Nrhop}
and 
\eqref{pi1min},
\begin{multline*}
  n\ge  \ell(N') \ge \pi_1(\xi)-\xi \ge \xi^2/(2\log \xi) - \xi
  \ge x_6^2/(2\log x_6)-x_6 \\
  \ge \nu_1 \stackrel{def}{=\!=} 10^{36}.
\end{multline*}
From \eqref{xinM}, by setting $L_1=\log \nu_1=82.89\ldots$, $\la_1=\log L_1=4.41\ldots$,
\[\xi\le  \sqrt{n\log n}\left(1+\frac{\la-1}{2L}\right)
\le \sqrt{n\log n}\left(1+\frac{\la_1-1}{2L_1}\right) \le 1.021 \sqrt{n\log n},\]
and, from \eqref{dusart3} with $\al=0.15$ and \eqref{lognxi},
\begin{multline}
  \label{logg2}
 \theta(\xi)-\xi \le \frac{0.15\,\xi}{\log^3 \xi}\le
   \frac{1.2\,\xi}{\log^3 n}\le \frac{1.2\times
     1.021\,\la^2}{L^2}\sqrt{n\log n}\frac{1}{L\la^2}\\
\le \frac{1.23\,\la^2}{L^2}\sqrt{n\log n}\frac{1}{L_1\la_1^2}
\le \frac{8\,\la^2}{10^4\,L^2} \sqrt{n\log n}.
\end{multline}
We also have
\begin{multline}
  \label{logg3}
 0.75(n\log n)^{1/4}=\frac{\la^2}{L^2}\sqrt{n\log n}
\left(\frac{0.75\,L^{7/4}}{n^{1/4}\la^2}\right)\\
\le \frac{\la^2}{L^2}\sqrt{n\log n}
\left(\frac{0.75\,L_1^{7/4}}{\nu_1^{1/4}\la_1^2}\right)
\le \frac{9\,\la^2}{10^8\,L^2}\sqrt{n\log n}.
\end{multline}
Finally, from \eqref{xinM}, \eqref{logg1}, \eqref{logg2} and \eqref{logg3},
\begin{align}\label{minzn}
\log g(n) &\le \theta(\xi)+0.75(n\log n)^{1/4}
\le \xi+\sqrt{n\log n} \,\frac{\la^2}{L^2}
  \left(\frac{8}{10^4}+\frac{9}{10^8}\right)  \notag\\
&\le \sqrt{n\log n}\left(1+\frac{\la-1}{2L}-\frac{\la^2}{L^2}
  \left(\frac{13-8-0.0009}{10^4}\right)\right)\notag\\
&\le \sqrt{n\log n}
  \left(1+\frac{\la-1}{2L}-\frac{4\,\la^2}{10^4\,L^2}\right),
\end{align}
which completes the proof of \eqref{maxgn} for $n \ge \nu_0$.
\end{itemize}
\qed

\subsection{Proof of the upper bound  \eqref{maxgn}
for $n < \nu_0$.}\label{parproof174}
The inequality \eqref{maxgn} for $4\le n < \nu_0$ will follow from the
lemma:

\begin{lem}\label{lemzn}
For $4\le n \le \nu_0$, $z_n$ defined by \eqref{zn} satisfies
\begin{multline}\label{n545}
z_6=3.18\ldots \ge z_n \ge z_{\nu_2}=0.005\,455\,048\,036\ldots >0\\
\qtx{with} \nu_2=6\,473\,549\,497\,145\,122.
\end{multline}  
\end{lem}

\begin{proof}
For $4\le n \le 18$, we calculate $z_n$ and obtain
$z_{12}=1.73\ldots \le z_n \le z_6=3.18\ldots$
For $n\ge 19$, we compute $z_{\ell(N)}$
for all superchampion numbers $N$ satisfying $19\le \ell(N) \le \nu_0$. 
The minimum is attained in $\nu_2$ and the maximum is $z_{19}=1.53\ldots$,
which, by applying Lemma \ref{lemanzn} (ii), completes the proof of \eqref{n545}.
We have $z_2=-2.05\ldots$ and $z_3=-2.38\ldots$ It is possible that
$z_n \ge z_{\nu_2}$ holds for all $n\ge 4$ but we have not been able to
prove it. We have only proved, from \eqref{minzn} and \eqref{n545},
that $z_n \ge 0.0004$ holds for $n \ge 4$. 
\end{proof}

\section{Study of  $g(n)/h(n)$ for large $n$'s }\label{parLargen}

\subsection{Effective estimates of $\log g(n) -\log h(n)$}
\label{parEffEstgsh}

\begin{prop}\label{propgshlarge}
If $n\ge \nu_0$ (defined by \eqref{n0}), we have
\begin{multline}
  \label{gshn}
\frac{\sqrt 2}{3}(n\log n)^{1/4}\left(1+
  \frac{\log \log n-11.6}{4\log n} \right)\\
\le
\log \frac{g(n)}{h(n)}\le \frac{\sqrt 2}{3}(n\log n)^{1/4}
\left(1+\frac{\log \log n+2.43}{4\log n}\right) .
\end{multline}
\end{prop}

\begin{proof}
  For $n\ge \nu_0$, we consider the two superchampion numbers $N'$ and
  $N''$ and $\xi$ defined in Definition \ref{nsuperch}. From
  \eqref{N''xiN'}, we have $N''\le \xi N'$ and from \eqref{N'n}, $N'\le
  g(n) < N''$.

In view of estimating $h(n)$, we need the value of $k=k(n)$ defined by
\eqref{kn}. For that, we have to convert the additive excess $E(N')$
(cf. \eqref{E}) in large primes. More precisely, if $p_{i_0}$ denotes the largest
prime factor of $N'$ and $\si_{i_0}=\sum_{p\le p_{i_0}} p$ (cf. \eqref{Nj}, 
from \eqref{Nrhop} and \eqref{pi0p1>xi}, we have 
$\sum_{p\mid N'} p=\sum_{p<\xi} p=\si_{i_0}$ and from \eqref{ellN'E}, 
$\ell(N')-E(N')=\si_{i_0}$ so that, from the definition \eqref{defs} of
$s=s(n)$,
\begin{equation}
  \label{sigi0n}
  \si_{i_0+s}\le n < \si_{i_0+s+1}
\end{equation}
and, from \eqref{kn}, $k=k(n)=i_0+s$. As $h$ is nondecreasing on $n$, from
\eqref{hsigk}, one deduces
$$h(\si_{i_0+s})=N_{i_0+s}\le h(n) \le N_{i_0+s+1}=h(\si_{i_0+s+1})$$
and
\begin{equation}
  \label{gsh1}
  \frac{N'}{N_{i_0+s+1}}\le \frac{g(n)}{h(n)}\le
    \frac{N''}{N_{i_0+s}}\le \frac{\xi N'}{N_{i_0+s}}.
\end{equation}

\bfni{The lower bound.}  Observing from \eqref{ellN'E*} that 
$\log N'=\sum_{p\mid N'} \log p+E^*(N')=\log N_{i_0}+E^*(N')$, from \eqref{gsh1},
\eqref{EN*} and \eqref{sp1log}, one gets
\begin{multline}\label{gsh2}
\log\frac{g(n)}{h(n)} \ge
  \log \frac{N'}{N_{i_0+s+1}}=E^*(N')-\sum_{i=i_0+1}^{i_0+s+1} \log p_i \ge 
E^*(N')-(s+1)\log p_{i_0+s+1}\\
\ge \sqrt{\frac{\xi}{2}}\left(1-\frac{0.521}{\log \xi}-\frac 13
      -\frac{0.345}{\log \xi}\right)=
\frac{\sqrt{2\xi}}{3}\left(1-\frac{1.299}{\log \xi}\right).
\end{multline}
Now, as the third derivative of $u\mapsto \sqrt{1+u}$ is positive,
from Taylor's formula and  Corollary \ref{corominxi}, 
\begin{equation}
  \label{sqrxin}
  \sqrt \xi \ge (n\log n)^{1/4}\left(1+\frac u2-\frac{u^2}{8}\right)
\qtx{with} u=\frac{\log\log n-1.019}{2\log n}.
\end{equation}
By writing $L$ for $\log n$, $\la$ for $\log \log n$, $L_0$ for $\log
\nu_0$ and $\la_0$ for $\log\log \nu_0$, from \eqref{fab},
\[\frac{u^2}{8}=\frac{(\la-1.019)^2}{4\times 8 \,L^2}\le
\frac{(\la_0-1.019)^2}{4\times 8\,L_0 L}\le \frac{0.022}{4L},\]
\[1+\frac u2-\frac{u^2}{8}\ge 1+\frac{\la-1.019}{4L}-\frac{0.022}{4L}=
 1+\frac{\la-1.041}{4L}\]
so that, from \eqref{sqrxin}, one gets
$\sqrt \xi \ge (n\log n)^{1/4}(1+(\la-1.041)/(4L))$. Further, from \eqref{lognxi},
 $1.299/\log \xi < 10.392/(4L)$ holds and \eqref{gsh2}
implies 
\begin{align*}
  \log\frac{g(n)}{h(n)} &\ge \frac{\sqrt 2}{3}(n\log n)^{1/4}
\left(1+\frac{\la-1.041}{4L}\right)  \left(1-\frac{10.392}{4L}\right)\\
&\ge \frac{\sqrt 2}{3}(n\log n)^{1/4}
\left(1+\frac{\la-11.433}{4L}-\frac{10.392(\la_0-1.041)}{(4L_0)(4L)}\right)  \\
&\ge \frac{\sqrt 2}{3}(n\log n)^{1/4}\left(1+\frac{\la-11.6}{4L}\right)
\end{align*}
which proves the lower bound of \eqref{gshn}.

\bfni{The upper bound.} Similarly, from \eqref{gsh1}, 
\eqref{pi0p1>xi}, \eqref{EN*} and \eqref{sm1>}, we have
\begin{multline}\label{gsh3}
  \log\frac{g(n)}{h(n)} \le
  \log \frac{\xi N'}{N_{i_0+s}}=\log \xi + E^*(N')-\sum_{i=i_0+1}^{i_0+s} \log p_i \le 
E^*(N')-(s-1)\log \xi\\
\le \sqrt{\frac{\xi}{2}}\left(1+\frac{0.305}{\log \xi}-\frac 13
      -\frac{0.0724}{3 \log \xi}\right)=
\frac{\sqrt{2\xi}}{3}\left(1+\frac{0.4213}{\log \xi}\right).
\end{multline}
Further, by using the inequality $\sqrt{1+t}\le 1+t/2$, it follows
from \eqref{minxicoro} and \eqref{lognxi} that
\begin{align*}
\log\frac{g(n)}{h(n)} &\le \frac{\sqrt 2}{3}(n\log n)^{1/4}
\left(1+\frac{\la-1}{4L}\right)  \left(1+\frac{3.3704}{4L}\right)\\
&\le \frac{\sqrt 2}{3}(n\log n)^{1/4}
\left(1+\frac{\la+2.3704}{4L}+\frac{3.3704(\la_0-1)}{(4L_0)(4L)}\right)  \\
&\le \frac{\sqrt 2}{3}(n\log n)^{1/4}\left(1+\frac{\la+2.43}{4L}\right)
\end{align*}
which ends the proof of Proposition \ref{propgshlarge}.
\end{proof}

\subsection{Asymptotic expansion of $\log g(n)-\log h(n)$}\label{parasy}

\begin{prop}\label{propasy}
Let $n$ be an integer tending to infinity. $N',N''$, $\xi$  and
$\xi_2$ are
defined by Definition \ref{nsuperch}.
Then, for any real number $K$, when $n$ and $\xi$ tend to infinity, we have
\begin{equation}
  \label{asyxi}
\log \frac{g(n)}{h(n)}=\left( \xi_2-\li(\xi_2^3)\frac{\log
    \xi}{\xi}\right)\left(1+\co_K\left(\frac{1}{\log^K \xi}\right)\right).
\end{equation}
\end{prop}

\begin{proof}
The proof follows the lines of the proof of Proposition \ref{propgshlarge}. 
Let $K$ be a real number as large as we wish.
First, from the Prime Number Theorem, for $r \ge 0$, it is easy to deduce
\begin{equation}
  \label{pirasy}
 \pi_r(x)=\sum_{p\le x} p^r =(\li(x^{r+1})) (1+\co(1/\log^K x) ),
 \qquad x\to \iy.
\end{equation}
From Proposition \ref{propx2}, when $\xi\to \iy$, we know that
\begin{equation}
  \label{xi2equiv}
  \xi_2\sim \sqrt{\xi/2}
\end{equation}
and, from Proposition \ref{propsn}, that
\begin{equation}
  \label{sequiv}
  s\sim \frac{\sqrt \xi}{3\sqrt 2 \log \xi}.
\end{equation}
Note that \eqref{sequiv} implies
\begin{equation}\label{spm}
s\pm 1=s(1+\co(1/\log^K \xi) ).
\end{equation}
By using the crude estimate $\pi_j(\xi_j)\le \xi_j^{j+1}$,
from \eqref{EN1}, \eqref{xk1sk} and \eqref{J}, it follows that
\begin{multline}
  \label{ENasy}
E(N')\le \sum_{j=2}^J \pi_j(\xi_j)\le 
\pi_2(\xi_2)+\sum_{j=3}^J \xi_j^{j+1}\le
\pi_2(\xi_2)+\sum_{j=3}^J \xi^{1+1/j}\\
\le
\pi_2(\xi_2)+J\,\xi^{4/3}= \pi_2(\xi_2)+\co(\xi^{4/3}\log \xi).  
\end{multline}
In the same way, from \eqref{ENmin}, one gets 
$E(N')\ge \pi_2(\xi_2)+\co(\xi)$ which, together with   
\eqref{ENasy}, \eqref{pirasy} and \eqref{xi2equiv}  yields
\begin{equation}
  \label{ENxi2}
E(N') =(\li(\xi_2^3)) (1+\co(1/\log^K \xi_2) ) =(\li(\xi_2^3)) (1+\co(1/\log^K \xi) ). 
\end{equation}
From \eqref{E*N1}, similarly, we have
\begin{equation}
  \label{E*Nthxi2}
  E^*(N')=\theta(\xi_2)+\co(J\,\xi_3)=\xi_2 (1+\co(1/\log^K \xi) ).
\end{equation}
It follows from Lemma \ref{lempixy} and \eqref{sequiv}
that the number of primes between
$\xi$ and $\xi(1+1/\log^K \xi )$ satisfies, for $n$ and $\xi$ large enough, 
$\pi(\xi(1+1/\log^K \xi)) -\pi(\xi)>\xi/(2\log^{K+1} \xi) > \sqrt \xi > s+1$, 
which, via \eqref{pi0p1>xi}, \eqref{defs} and \eqref{ellN'N''}, implies
\begin{equation}
  \label{pi0sp1eq}
 \xi \le p_{i_0+1} \le p_{i_0+s+1} \le  \xi (1+\co(1/\log^K \xi) )
\end{equation}
and 
\begin{align*}
  (s-1)\xi &\le p_{i_0+1}+\ldots +p_{i_0+s} -\xi\le
             n-\ell(N')-\xi+E(N') \\
 & \le   \ell(N'')-\ell(N')-\xi+E(N' \le E(N'))\\
&\le n-\ell(N')+E(N') \le p_{i_0+1}+\ldots +p_{i_0+s+1}\\
 &\le (s+1)p_{i_0+s+1} \le (s+1)\,\xi\, (1+\co(1/\log^K \xi) ).
\end{align*}
From \eqref{spm}, it follows that $E(N')=s\xi(1+\co(1/\log^K \xi))$ and,
from \eqref{ENxi2} ,
\begin{equation}
  \label{sxi3}
  s=\frac{\li(\xi_2^3)}{\xi}\left(1+\co\left(\frac{1}{\log^K \xi}\right)\right).
\end{equation}
From \eqref{gsh2} and \eqref{gsh3}, we have
\begin{equation}
  \label{gshasy}
E^*(N')-(s+1)\log p_{i_0+s+1} \le \log \frac{g(n)}{h(n)} \le
E^*(N')-(s-1)\log \xi, 
\end{equation}
which, from \eqref{spm}, \eqref{pi0sp1eq}, \eqref{sxi3} and
\eqref{E*Nthxi2}, proves \eqref{asyxi}.
\end{proof}

\section{Proof of Theorem \ref{thmgsh}}\label{parproofthmgsh}

\subsection{Proof of Theorem \ref{thmgsh} (i).}\label{parproofthmgshi} 

We assume that $n$ tends to infinity. $N',N''$ and $\xi$ are defined
by \eqref{N'n}.
From Proposition \ref{propasy}  one deduces
\begin{equation}
    \label{gshasyxi}
\log \frac{g(n)}{h(n)}\asymp \xi_2-\li(\xi_2^3)\frac{\log \xi}{\xi}
=\frac{\sqrt{2\xi}}{3} F \qtx{with} F=\frac{3}{\sqrt{2\xi}}
\left(\xi_2-\li(\xi_2^3)\frac{\log \xi}{\xi}\right).
\end{equation}
By using \eqref{x2asy} and \eqref{lixinfini}, we get
the asymptotic expansion of $F$ in terms of $t=1/\log \xi$
(cf. \cite{web})
\begin{equation}
    \label{F}
F=F(t)=1-\frac{2+3\log 2}{6}t
-\frac{32+48\log 2+9\log^2 2}{72}t^2 + \ldots
\end{equation}

From 
\eqref{Nrho}, \eqref{N'n},  \eqref{N''xiN'},
\eqref{ellN'E*} an
d \eqref{EN*}, we have
\begin{multline*}
  \theta^-(\xi)\le \log N' \le \log g(n) \le \log N'' \\
  \le \log N' + \log \xi
= \theta^-(\xi)+E^*(N')+\log \xi=\theta(\xi)+\co(\sqrt \xi)
\end{multline*}
so that, from the Prime Number Theorem and \eqref{g=li}, 
for any real number $K$, we have
\begin{equation}
  \label{xisqr}
  \xi=\sqrt{ \limm n}\ (1+\co_K(1/\log^K n))
\end{equation}
that we write $\xi \asymp \sqrt{ \limm n}$.
Therefore, from \eqref{gshasyxi} and \eqref{asygn}, we can get the
asymptotic expansion of $\log(g(n)/h(n))$.
More precisely, we may use Theorem 2 of \cite{Sal} to get
\begin{equation}
  \label{Pj}
 \log \frac{g(n)}{h(n)} \asymp  \frac{ \sqrt 2}{3} (n\log n)^{1/4}
\left(1+\sum_{j\ge 1} \frac{P_j(\log \log n)}{\log^j n} \right)
\end{equation}
where $P_j$ is a polynomial of degree $j$ satisfying the 
induction relation
\begin{equation}
  \label{Pjind}
  \frac{d}{dt} (P_{j+1}(t)-P_j(t))=\left(\frac 14-j\right) P_j(t).
\end{equation}
For that, one sets $y=\log(\limm (n))$, $n=\li(e^y)$, $\xi \asymp e^{y/2}$
and, from \eqref{lixinfini}, $n\asymp \frac{e^y}{y}\sum_{k\ge 0} \frac{k!}{y^k}$
so that, from \eqref{gshasyxi},
\[\log \frac{g(n)}{h(n)} \asymp  \frac{ \sqrt 2}{3}e^{y/4}F(2/y)\]
holds. Finally, we apply the procedure {\sl theorem2\_part2} of
\cite[p. 234]{Sal} with $\al=1,\beta=1/4,\ga=0, G(t)=F(2t)$
(with $F$ defined by \eqref{F}), $d(t)=\sum_{k\ge 0} k! t^k$ and $x=n$.
\qed

The values of the polynomials $P_j$ can be found on the
website \cite{web} (cf. \cite{RobB}  and \cite{MR96} for similar results).

\subsection{Proof of Theorem \ref{thmgsh} (ii) and (iii).}\label{parproofthmgshii}
 
 Let us define $\beta_n$ as the unique number such that
 \begin{equation}\label{defbn}
 \log \frac{g(n)}{h(n)}
   = \frac{\sqrt 2}{3}(n\log n)^{1/4} \left(1+\frac{\log \log n + \beta_n}{4 \log n}\right)
   \end{equation}
 An easy computation gives
 \begin{equation}\label{bnvalue}
 \beta_n = 
 6 \sqrt2 (\log n)^{3/4}  \frac{\log g(n)-\log h(n)}{n ^{1/4}} 
-4\log n  - \log\log n.
\end{equation}

From the non decreasingness of the functions $\log$, $\log\log$, $g$
and $h$ we deduce from \eqref{bnvalue} the following

\begin{lem}
  Let $n_1 \le n_2$ be two integers. Then, for every $n \in \undeux$,
  \begin{equation}\label{majbn}
    \beta_n \le 6\sqrt2 (\log n_2)^{3/4}\
    \frac{\log g(n_2)-\log h(n_1)}{n_1^{1/4}} \
    - 4 \log n_1 - \log\log n_1
 \end{equation}
 and,  if  $g(n_1) \ge h(n_2)$ is satisfied
   \begin{equation}\label{minbn}
     \beta_n \ge 6\sqrt2 (\log n_1)^{3/4} \
     \frac{\log g(n_1)-\log h(n_2)}{n_2^{1/4}}\
     - 4 \log n_2 - \log\log n_2
  \end{equation}
  \end{lem}

  The expensive operations in the computation of the bounds given in
  \eqref{majbn} and \eqref{minbn} are the computations of
  $g(n_1), g(n_2), h(n_1), h(n_2)$.  In the particular case where
  $n_1, n_2 = \ell(N_1), \ell(N_2)$ for two consecutive superchampions
  $N_1, N_2$, we will use the following lemma to
  quickly bound $\beta_n$ on the slice $\undeux$.
 
 \begin{lem}
   Let $n_1= \ell(N_1)$,  $n_2= \ell(N_2)$ where $N_1$, $N_2$ are two
   consecutive superchampions, $k_1= k(n_1)$
 (resp. $k_2=k(n_2)$), $m_1= n_1 - \sigma_{k_1}$
 (resp. $m_2= n_2 - \sigma_{k_2}$) and $q$ the  first prime
 not smaller than  $p_{k+1} - m_1$.
 Then, for every $n$ in $[n_1,\, n_2]$,
 \begin{equation}\label{quickmajbetan}
   \beta_n \le 6\sqrt2 (\log n_2)^{3/4}\
    \frac{\log N_2 - \theta(p_{k_1+1}) + \log q}{n_1^{1/4}} \
    - 4 \log n_1 - \log\log n_1
 \end{equation}
 and, if $\log N_1 >  \theta(p_{k_2+1}) - \log (p_{k_2+1} - m_2)$
 \begin{equation}\label{quickminbetan}
   \beta_n \ge
   6\sqrt2 (\log n_1)^{3/4} \
     \frac{\log N_1-\theta({p_{k_2+1}}) + \log (p_{k_2+1} - m_2)}{n_2^{1/4}}\
     - 4 \log n_2 - \log\log n_2.
 \end{equation}
 \end{lem}

 \begin{proof}
\suppress{
   by using $\log g(n_1) = \log N_1$ (resp.
  $\log g(n_2) = \log N_2$), and by replacing the value $\log h(n_1)$,
  (resp. $\log h(n_2)$) by the bound given in \eqref{sliceminh}
  (resp. \eqref{slicemajh}),
}
   \begin{itemize}
   \item
    We get \eqref{quickmajbetan} from \eqref{majbn} by noticing that $\log g(n_2) = \log
    N_2$ and by using \eqref{sliceminh} to minimize $h(n_1)$.
  \item
     We get \eqref{quickminbetan}  from \eqref{minbn}  by noticing that $\log g(n_1) = \log
    N_1$ and by using \eqref{slicemajh} to maximize $h(n_2)$.
   \end{itemize}
 \end{proof}
 
\paragraph{Proof of Theorem 1.5.(ii).}
Considering \eqref{defbn} we have to prove that $\beta_n \le 2.43$ for
$ n > \nu_5 =  3\, 997\, 022\, 083\, 662$.
For $n>\nu_0$, it results of Proposition \ref{propgshlarge}.

We first enumerate all the pairs of consecutive superchampions
$\le N'_0$, and for each of these pairs we compute the upper bound
given by \eqref{quickmajbetan}.  It appears that for
$n \ge \nu_3 = 23\,542\,052\,569\,006 $, $\beta_n < 2.43$.
 To get the largest $n$ which does not satisfy Theorem 1.5.ii  we use the
dichotomic procedure \texttt{ok\_rec} described in Section \ref{pardicho}
on the interval $[2,\nu_3]$,
choosing the functions $ok(n)$ which returns true if and only if
 $\beta_n \le 2.43$,  and the function $good\_interval(n_1,n_2) $  which returns true if and only if
 the right term of \eqref{quickmajbetan} is not greater than $2.43$.
 The call \texttt{ok\_rec(2, {${\mathbf\nu_3}$})}
 gives  $\nu_5$ as the largest number $n$ such that 
$\beta_n = 2.430\ 001\ 869\ldots > 2.43$.  This
computation generated $2$ calls of \texttt{ok(n)}.
and  $5\,017\,255$ calls of \texttt{good\_interval},
and it took 40h.
\qed

\paragraph{Proof of Theorem 1.5.(iii).}
\label{parproofthmgshiii}
For $n > \nu_0$,  it results  of Proposition \ref{propgshlarge}.

As in the previous paragraph, we first enumerate all the pairs of
consecutive superchampions $N_1, N_2 \le N'_0$. We have checked that, for $\ell(N_1) \ge 1487$,
$\log N_1 > \theta(p_{k_2+1}) - \log(p_{k_2+1} - m_2 )$, so that
\eqref{quickminbetan} holds, and then, we verify that
for $n \ge \nu_4 = 1\,017\,810$, $\beta_n > -11.6$ holds.

Now the call \texttt{ok\_rec(1487, $\mathbf\nu_4$)}with the function
$good\_interval(n_1,n_2) $  which returns true if and only if
the right term of \eqref{quickminbetan} is smaller than $-11.6$
and the function $ok(n)$ which returns true if and
only if $\beta_n > -11.6$, we get $n = 4\,229$ as
the largest $n$ sucht that $\beta_n < -11.6$.

\suppress{
For $2\le n < \nu_4$,  one uses the naive algorithm to calculate $g(n)$ 
(cf. \cite[Section 2]{DNZ}), $h(n)$ (cf. \cite[Section 1.4]{DNh1}).
The smallest value of $\be_n$ is  $\be_{160}=-20.718\ldots$.
The largest $n$ for which  $\be_n< -11.6$ holds is $4\,229$.
}

\qed
\subsection{Proof of Theorem \ref{thmgsh} (iv).}
\label{parproofthmgshiv} 

The inequality $g(n)\ge h(n)$ follows from \eqref{g} and \eqref{h}.
For $ n \ge 4\,230$, inequality $g(n) > h(n)$ is an easy consequence
of point (iii).
We end the proof by computing $g(n)$ and $h(n)$ for $1\le n < 4\,230$.
\qed

\subsection{Proof of Theorem\ref{thmgsh} (v).}\label{parproofthmgshv}

Let  $d_n$ defined by
\begin{equation}
  \label{dn}
d_n=b_n-a_n=\frac{\log g(n)-\log h(n)}{(n\log n)^{1/4}}=
\frac{\sqrt 2}{3}\left(1+\frac{\log \log n+\beta_n}{4\log n}\right).  
\end{equation}

For $n > \nu_5$, from point (ii), we have
\[\log \frac{g(n)}{h(n)} \le \frac{\sqrt 2}{3} (n\log n)^{1/4}
\left(1+\frac{\log \log \nu_5+2.43}{4 \log \nu_5}\right) \le 
0.5\, (n\log n)^{1/4}.\]

By the non-decreasingness of $g$ and $h$, if $n_1 \le n_2$,
$d_n$ is bounded above on $\undeux$ by $M(n_1,n_2)$, with
\begin{equation}\label{boundn}
   M(n_1, n_2) =  (\log g(n_2) - \log h(n_1)(n_1 \log n_1)^{1/4}.
 \end{equation}
Thus, the inequality
\begin{equation}\label{gd5}
  M(n_1, n_2) < 0.62
\end{equation}
is a sufficient condition ensuring that $d_n < 0.62$ on the whole
interval $\undeux$.

As in paragraph \ref{parproofthmgshii}, 
for all the pairs $n_1= \ell(N_1)$, $n_2= \ell(N_2)$ where $N_1$, $N_2$
are two consecutive superchampions with $\ell(N_2) \le \nu_5$ we quickly get
an upper bound of $M(n_1, n_2)$ by using
$g(n_2)= \log(N_2)$ and bounding below $\log h(n_1)$ by \eqref{sliceminh}.  It
appears that this bound is smaller than $0.62$ for $n \ge 49\,467\,083$.

Now the call \texttt{ok\_rec(2, 49467083)} using $ok(n)$ which returns
true if and only if $d_n < 0.62059$
and $good\_interval(n_1, n_2)$  which returns true if and
only $M(n_1, n_2) < 0.62059$, gives us the last value of $d_n$ which is
greater than $0.62059$, this value is $d_{2243} = 0.620\,665\,265\,68...$ Note
that $g(2243)$ is a superchampion number associated to
$\rho=139/\log 139$ and $149/\log 149$.
Finally, by computing $g(n)$ and $h(n)$, we checked that $d_n < 0.62$
holds
for $2 \le n <  2243$.
\qed

\section{Proof of Theorem \ref{thmgnHR}}\label{parproofthmgn}

\subsection{Proof of Theorem \ref{thmgnHR} (i).} 

For $n\ge 2$, the point (i) of Theorem \ref{thmgnHR} follows from the definition
\eqref{an} of $a_n$ and from the point (iv), below. For $n=1$,
$\limm(1)=1.96\ldots$  and $g(1)=1$ so that 
$\log g(1) < \sqrt{\limm (1)}$ holds.
\qed

\subsection{Proof of Theorem \ref{thmgnHR} (ii).} 

From now on, the following notation is used : $L=\log n$, $\la=\log
\log n=\log L$ and $\la_0=\log \log \nu_0$.

From \eqref{bn}, for $n \ge \nu_0$ (defined by \eqref{n0}), we have
\[\log g(n)=\log h(n)+\log \frac{g(n)}{h(n)}=-b_n(n\log
n)^{1/4}+\sqrt{\limm n} +\log \frac{g(n)}{h(n)}\]
and, from \eqref{an}, Theorem \ref{thmhnHR} (iii) and 
Proposition \ref{propgshlarge}, one gets
\begin{align*}
a_n &=\frac{\sqrt{\limm n}-\log g(n)}{(n\log n)^{1/4}}=
      b_n- \frac{\log (g(n)/h(n))}{(n\log n)^{1/4}}\\
  & \ge \frac 23 -c-\frac{0.23\,\la}{L}-
\frac{\sqrt 2}{3}\left(1+\frac{\la+2.43}{4L}\right)\\
& =\;\; \frac{2-\sqrt 2}{3} -c-\frac{\la}{L}
\left(0.23+\frac{\sqrt 2}{12}+\frac{2.43\sqrt 2/\la}{12}\right)\\
&\ge \frac{2-\sqrt 2}{3} -c-\frac{\la}{L}
\left(0.23+\frac{\sqrt 2}{12}+\frac{2.43\sqrt 2/\la_0}{12}\right)
\ge \frac{2-\sqrt 2}{3} -c-\frac{0.43\,\la}{L},
 \end{align*}
which proves  point (ii) for $n> \nu_0$. 

\begin{lem}\label{lemgncon1939}
For $2\le n\le \nu_0$, $a_n$ defined by \eqref{an} satisfies
\begin{equation}
  \label{an0}
   a_n \ge a_{6\, 473\, 580\, 667\, 603\, 736} = 0.193938608602\ldots.
\end{equation}
\end{lem}

\begin{proof}
  For $2\le n \le 42$, one checks that $a_n \ge 0.4$. 
For $43\le n< \nu_0$, by Lemma \ref{lemanzn} (i), the minimum is
attained in $\ell(N)$ with $N$ being a superchampion number satisfying
$43 \le \ell(N) \le \nu_0 = \ell(N'_0)$. So, by enumerating
$(N, \log N)=(N, \log g(\ell(n)))$ for $N \le N'_0$ we check that
the minimum is $0.193938602\ldots$, attained for
$\ell(N)= 6\, 473\, 580\, 667\, 603\, 736$.
\end{proof}

For $3\le n < \nu_0$, Lemma \ref{lemgncon1939} shows that $a_n \ge
(2-\sqrt 2)/3 -c=0.149\ldots$ holds, which, as $\log \log n$ is
positive, proves point (ii).
For $n=2$, $a_2=0.91\ldots$ and point (ii) is still satisfied.
\qed

\subsection{Proof of Theorem \ref{thmgnHR} (iii).}\label{thhmgnHR3}

\paragraph{For $\mathbf{n > \nu_0}$.}
From Theorem \ref{thmhnHR} (iv) and Proposition \ref{propgshlarge}, we have
\begin{eqnarray*}
a_n &=& b_n-\frac{\log (g(n)/h(n))}{(n\log n)^{1/4}} \le
\frac 23 +c+\frac{0.77\,\la}{L}-
\frac{\sqrt 2}{3}\left(1+\frac{\la-11.6}{4L}\right)\\
&\le& \frac{2-\sqrt 2}{3} +c+\frac{\la}{L}
\left(0.77-\frac{\sqrt 2}{12}+\frac{11.6\sqrt 2/\la_0}{12}\right)
\le \frac{2-\sqrt 2}{3} +c+\frac{1.02\,\la}{L}.
 \end{eqnarray*}

\paragraph{For $\mathbf n \le \nu_0$.}
Let us suppose $16 \le n_1 \le n_2$. The
non decreasingness of $\li$, $g$, the positivity of $a_n$ (cf.
\eqref{an0})
and therefore of $\sqrt{\li^{-1}(n)}-\log g(n)$ imply that, for  $n \in \undeux$,
\begin{equation}\label{anmaj}
  a_n = \frac{\sqrt{\li^{-1}(n)}-\log g(n)}{(n\log n)^{1/4}}
  \le  R(n_1, n_2) = \frac{\sqrt{\li^{-1}(n_2)}-\log g(n_1)}{(n_1\log n_1)^{1/4}}
\end{equation}
Since $n_1 \ge 16$, the function $\log\log n/\log n$ is decreasing on
$\undeux$, and, in view of \eqref{anmaj},
\begin{equation}\label{goodint}
 R(n_1, n_2) 
  \le  \frac{2-\sqrt2}{3} + c+ m \frac{\log\log n_2}{\log n_2}
\end{equation}
is a sufficient condition ensuring that, for all $n \in \undeux$,
 \begin{equation}\label{okn}
   a_n   <   \frac{2-\sqrt2}{3} + c + m \frac{\log\log n}{\log n}.
 \end{equation}
 In the case $n_1=\ell(N_1)$, $n_2=\ell(N_2)$, where $N_1, N_2$
 are consecutive superchampions, $g(n_1)=\log (N_1)$, and,
 by enumerating all the pairs of consecutive superchampions
 $\le N'_0$, we check that, if $m=1.02$,  inequality
 \eqref{anmaj} is satisfied for $n \ge 5\,432\,420$.
 To compute the largest $n$ which does not
 satisfy this inequality we call
 \texttt{ok\_rec(2, 5432420)} with the boolean fonction $ok(n)$ which returns true
 if and only if \eqref{okn} is true, and the procedure
 \texttt{good\_interval(n1, n2)} which returns true if and only
 \eqref{goodint} is satisfied.
This gives us $19424$ as the largest integer which does not satisfy
point (iii).
\qed

\subsection{Proof of Theorem \ref{thmgnHR} (iv).}\label{parproofgn4}

For $n\ge \nu_0$, from point (ii) it follows  that\\
\[
  a_n\ge\frac{2-\sqrt 2}{3}-c -\frac{0.43\log \log \nu_0}{\log \nu_0}
  =0.11104\ldots
 \]
while, by Lemma \ref{lemgncon1939}, $a_n \ge 0.1939$ for $n\le \nu_0$.

By computing $a_n$ for $2 \le n \le 19424$, it appears that $a_n < a_2=
0.9102\ldots$, while, for $n \ge 19425$, by point (iii) and the
decreasingness of $\log\log n/\log n$,
\dsm{a_n <  \frac{2-\sqrt 2}{3} + c +\frac{1.02\log \log 19425}{\log
    19425}= 0.477\ldots}
\qed

\subsection{Proof of Theorem \ref{thmgnHR} (v).} 

The point (v) of Theorem \ref{thmgnHR} follows from the points (ii)
and point (iii).
 \qed

\subsection{Proof of Theorem \ref{thmgnHR} (vi).} 

From \eqref{dn} and Theorem \ref{thmgsh} (i), we have
\[b_n-a_n=d_n=\frac{\log(g(n)/h(n))}{(n\log n)^{1/4}}
=\frac{\sqrt 2}{3}\left(1+\frac{\log \log n+\co(1)}{4\log n}\right)\]
 whence,  from Theorem \ref{thmhnHR} (vi),
\begin{align*}
 a_n &= b_n-d_n\\
&\le \left(\frac 23 +c\right)\left(1+\frac{\log \log n+\co(1)}{4\log n}\right)
  - \frac{\sqrt 2}{3}\left(1+\frac{\log \log n+\co(1)}{4\log n}\right)\\
&= \left(\frac{2-\sqrt 2}{3} +c\right)
  \left(1+\frac{\log \log n+\co(1)}{4\log n}\right),
\end{align*}
which proves the upper bound of (vi). The proof of the lower bound is similar.
\qed
 
\bfni {Acknowledgements.}  We thank very much Richard Brent for
 attracting  our attention on the problem studied in this article.

\medskip

\noindent
Marc Del\'eglise, Jean-Louis Nicolas,\\
Univ. Lyon, Universit\'e Claude Bernard Lyon 1, CNRS UMR 5208\\
Institut Camille Jordan, Math\'ematiques, B\^at. Doyen Jean Braconnier,\\
43 Bd du 11 Novembre 1918,\quad F-69622 Villeurbanne cedex, France.

\medskip
\noindent
\url{m.h.deleglise@gmail.com},\\
\url{http://math.univ-lyon1.fr/homes-www/deleglis/}\\

\noindent
\url{nicolas@math.univ-lyon1.fr},\\ 
\url{http://math.univ-lyon1.fr/homes-www/nicolas/}
\end{document}